\documentclass[a4paper]{article} 
\usepackage[left=2.3cm,right=2.3cm,top=2.9cm]{geometry}
\usepackage{authblk}
\usepackage{CJKutf8}
\usepackage{yfonts}
\usepackage{lipsum}
\usepackage[utf8]{inputenc}
\title{On the localization regime of high-dimensional directed polymers in time-correlated random field}
\author{Jiaming Chen\thanks{chen.jiaming@cims.nyu.edu}}
\affil{Courant Institute of Mathematical Sciences, New York University}
\date{\today}

\usepackage{graphicx}
\usepackage{amssymb}
\usepackage{amsmath}
\usepackage{amsthm}
\usepackage{tikz}
\usepackage{todonotes}
\usepackage{enumitem}
\usepackage{verbatim}

\usepackage{centernot}
\usepackage{mathtools}
\usepackage{ stmaryrd }

\usepackage{hyperref}
\hypersetup{allcolors=black,
pdftitle={Directed polymer in time-correlated field},
pdfpagemode=FullScreen}
\numberwithin{equation}{section}
\usepackage{bbm}

\usepackage{braket}
\usepackage{amsmath,amsthm,amssymb}
\usepackage{physics}
\usepackage{mathtools}
\usepackage{bm}
\usepackage{esvect}
\usepackage[english]{babel}
\usepackage{extarrows} 
\usepackage{mathrsfs}
\usepackage{esint}

\usepackage{setspace}
\usepackage{tocloft}
\usepackage{titlesec}
\titleformat{\subsection}[runin]
  {\normalfont\large\bfseries}{\thesubsection}{1em}{}

\usepackage{amsmath}

\numberwithin{equation}{section}
\usepackage{bbm}

\usepackage{braket}
\usepackage{amsmath,amsthm,amssymb}
\usepackage{physics}
\usepackage{mathtools}
\usepackage{bm}
\usepackage{esvect}
\usepackage[english]{babel}

\newtheorem{theorem}{Theorem}[section]
\newtheorem{lemma}[theorem]{Lemma}

\theoremstyle{definition}
\newtheorem{definition}[theorem]{Definition}

\theoremstyle{remark}

\usepackage[english]{babel}

\DeclareMathOperator*{\esssup}{ess\,sup}

\newcommand\normx[1]{\lVert#1\rVert}
\newcommand\normy[1]{\big\lVert#1\big\rVert}

\makeatletter
\renewenvironment{proof}[1][\proofname]{%
  \par\pushQED{\qed}\normalfont%
  \topsep6\p@\@plus6\p@\relax
  \trivlist\item[\hskip\labelsep\bfseries#1\@addpunct{.}]%
  \ignorespaces
}{%
  \popQED\endtrivlist\@endpefalse
}
\makeatother
\begin{document}
\maketitle

\begin{abstract}
    This paper describes directed polymer on general time-correlated random field. Law of large numbers, existence and smoothness of limiting free energies are proved at all temperature. We also display the delocalized-localized transition, via separating techniques for entanglement of the random field.
\end{abstract}

\begin{spacing}{0} 
    \tableofcontents
\end{spacing}

%\clearpage

%-----------------------------------------
%-----------------------------------------
\section{Introduction}
    Take $N$ particles in a fluid and assume their interactions are connected by harmonic strings. With external forces and thermal fluctuation, the shape of these particles should be understood as a random configuration. This is a very primitive model for a \textit{polymer chain} consisting of $N$ particles wafting in water \cite{Smith/Bertola}. Admitting the effect of the external water molecules that randomly kick the particles making up our string, we furthermore allow that these kicks occur randomly and could correlate in both space and time.\par
    In the framework of statistical mechanics \cite{Huse/Henley}, the question we address here is: How does the stochastic impurities affect the macroscopic shape of the polymer chain? In this work, we try to answer this question in discrete models and in the perspective of the localization regime. The classical tréatise usually suppresses the entanglement and interactions among particles; But we find this compromising too simplified and we novelly take into account the space and time-correlations among the underlying random field.\par
    As in classical modeling \cite{Comets2}, we shall represent the polymer chain as a $N$-particle graph $\{(j,x_j)\}_{j=1}^N$ in $\mathbb{N}\times\mathbb{Z}^d$ so that the polymer configuration lives in $(d+1)$-dimensional lattice, and stretches dependently in time direction. Each point $(j,x_j)\in\mathbb{N}\times\mathbb{Z}^d$ on the graph stands for the position of $j$th particle in the chain. And we also assume that the transversal motion $(S_j)_{j=1}^N$ performs a finite-range simple random walk for consecutive particles in the chain taking all possible configurations at a fixed distance one from another.\par
    The complexity of the underlying random field is nonetheless a difficult topic where classical martingale structure \cite{Albeverio/Zhou} is destroyed and has been avoided by most literature on directed polymers. Indeed, pioneering work on the subtle non-i.i.d. environment include \cite{Medina/Hwa/Kardar/Zhang} where the nonlinear diffusion with correlated space-noise is discussed in the context of random interface growth; And in \cite{Rovira/Tindel} the authors introduced the Brownian polymer in space-correlated Gaussian field. Furthermore, the superdiffusivity is investigated in \cite{Bezerra/Tindel/Viens,Lacoin} in the same model. That said, there has not been much discussion on the challenging time-correlated scenario.\par
    Following previous work \cite{Rang/Song/Wang} on the scaling limits of partition function with omitted space-correlation \cite{Magdziarz/Szczotka/Zebrowski} and specific Gaussian dependence in time direction, we novelly show the existence and transition between localization and delocalization regimes where the polymer lives in transient dimensions, in the presence of both more general space and time-correlation. Heuristically, striking results characterizing the delocalized-localized transition under i.i.d. random field have been given in \cite{Carmona/Hu,Comets/Shiga/Yoshida}. Intuitively, delocalization implies that the polymer chain behaves like $(S_n)_{n\geq1}$ in transient dimensions, wheres the localization means that the polymer is extremely affected by its favorable medium and thus concentrates in just a few coordinates.\par
    Looking beyond, we should also remark that polymers can be defined with long-range random walks. See for instance \cite{Miura/Tawara/Tsuchida} where the law of increments belongs to the domain of attraction of an $\alpha$-stable law; And \cite{Viveros} where they discussed phase transitions when the random walk has very heavy tails. I also think it is possible to extend my techniques to long-range model of polymer chains. But one should pay particular notice to the specific cone-mixing structure in Section \ref{sec: auxiliary fields}. Perhaps a superposition of countably many cones suffices. Nonetheless, the picture is still obscure, and I encourage researchers (including myself) to future studies. \par\noindent
    \textbf{Acknowledgment}. The author wishes to thank his PhD advisor Prof. Dr. Alejandro Ramírez at NYU Shanghai for pointing out the notion of time-correlation, and for reading this manuscript as well as raising valuable comments.

%-----------------------------------------
%-----------------------------------------
\section{Directed polymer}
    A random walk on $\mathbb{Z}^d$, $d\geq1$ is a sequence $(S_n)_{n\geq0}$ starting from $S_0\in\mathbb{Z}^d$ and moves over the lattice in discrete time. Letting $\mathcal{P}(\mathbb{Z}^d)$ denote the canonical product Borel $\sigma$-algebra, we can alternatively describe its trajectories in the path space $((\mathbb{Z}^d)^{\mathbb{N}},\mathcal{P}(\mathbb{Z}^d)^{\otimes\mathbb{N}})$ via a probability measure $P^S_x$ whenever $S_0=x$.\par
    Throughout this paper, $(S_n)_{n\geq0}$ represents a random walk bounded i.i.d. increments, i.e.~$\normx{S_1}_1<\infty$. From now on, we consider only transient $(S_n)_{n\geq0}$, i.e.~$d\geq3$. Independently, we introduce a time-correlated Markovian random field $\omega=(\omega_{n,z}:\,n\in\mathbb{N},z\in\mathbb{Z}^d)$ with its law $\mathbb{P}$ and product $\sigma$-algebra $\mathscr{F}_\Omega$ on $\mathbb{N}\times\mathbb{Z}^d$ satisfying
    \begin{equation}\label{eqn: finit exponential moment of omega}            
        \mathbb{E}\big[\exp(\beta\omega_{n,z})\big]<\infty,\qquad\forall~\beta\in\mathbb{R},\quad n\in\mathbb{N}\quad\text{and}\quad z\in\mathbb{Z}^d.
    \end{equation}
    Here we adopt the following convention that $E^S_x\coloneqq E_{P^S_x}$ and $\mathbb{E}\coloneqq E_{\mathbb{P}}$. And $\beta$ refers to \textit{inverse temperature} in physics literature and will more or less reflect the level of disorder in our subsequent setting. And without loss of generality \cite[Remark 1.1]{Comets/Yoshida}, we will restrict to the case $\beta\geq0$.\par  
    Before stating the main model of this paper, we need first describe has exactly the time-correlation is defined. For this end, we need to define the distance between sets in the space-time lattice. For instance with $A,B\subseteq\mathbb{N}\times\mathbb{Z}^d$, we let $d_1(A,B)\coloneqq\inf\{\abs{x-y}_1+\abs{m-k}:\,(m,x)\in A,\,(k,y)\in B\}$ stand for the $\ell_1$-distance between subsets $A$ and $B$ of $\mathbb{N}\times\mathbb{Z}^d$. By saying $\omega$ Markovian, we mean 
    \[
            \mathbb{P}\big((\omega_{n,z})_{(n,z)\in V}\in\vdot\big|\mathscr{F}_{V^c}\big) = \mathbb{P}\big((\omega_{n,z})_{(n,z)\in V}\in\vdot\big|\mathscr{F}_{\partial V}\big),\qquad\forall~V\subseteq\mathbb{N}\times\mathbb{Z}^d,\quad\mathbb{P}\text{-a.s.,}
    \]
    where we follow the convention $\mathscr{F}_\Lambda\coloneqq\sigma(\omega_{n,z}:\,(n,z)\in\Lambda)$ for any $\Lambda\subseteq\mathbb{N}\times\mathbb{Z}^d$. The exact mixing nature of the time-correlated field is described as follows.

    \begin{figure}
    \centering
    \includegraphics[width=0.7\textwidth]{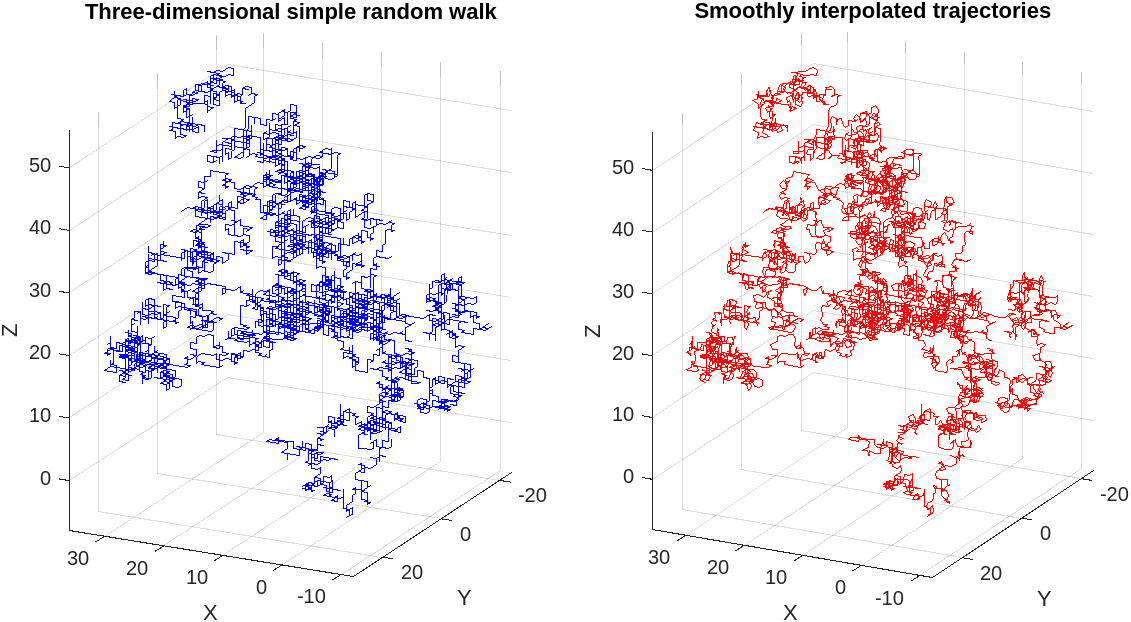}
    \caption{An illustration of simple random in transient dimensions.}
    \label{SRW}
    \end{figure}
    
    %definition: time-correlated condition
    \begin{definition}\label{def: time-correlated}
        Let $C$ and $g$ be given and fixed positive real numbers. A Markovian field $\omega$ with law $(\mathbb{P},\mathscr{F}_\Omega)$ on $\mathbb{N}\times\mathbb{Z}^d$ is said to satisfy the time-correlated condition $\textbf{(TC)}_{C,g}$ if for any finite subsets $\Delta\subseteq V\subseteq \mathbb{N}\times\mathbb{Z}^d$ with $d_1(\Delta,V^c)>1$ and $A\subseteq V^c$,
    \[
        \frac{d\mathbb{P}((\omega_{n,z})_{(n,z)\in\Delta}\in\vdot|\eta)}{d\mathbb{P}((\omega_{n,z})_{(n,z)\in\Delta}\in\vdot|\eta^\prime)} \leq \exp\bigg(C\sum_{(m,x)\in\partial\Delta,(k,y)\in\partial A}e^{-g\abs{x-y}_1-g\abs{m-k}}\bigg)
    \]
    simultaneously for all pairs of configurations $\eta,\eta^\prime\in\Omega\coloneqq([0,\infty))^{\mathbb{N}\times\mathbb{Z}^d}$ which agree on $V^c\backslash A$, $\mathbb{P}$-a.s. Notice that here we follow the convention
    \[
        \mathbb{P}\big((\omega_{n,z})_{(n,z)\in\Delta}\in\vdot\big|\eta\big) = \mathbb{P}\big((\omega_{n,z})_{(n,z)\in\Delta}\in\vdot\big|\mathscr{F}_{V^c}\big)~\text{given that}~(\omega_{n,z})_{(n,z)\in V^c}=\eta.
    \]
    \end{definition}
    Apart from this time-correlated condition $\textbf{(TC)}_{C,g}$, another type of correlation is incorporated in our discussion as well. The following correlation originates from X. Guo \cite{Guo} in the context of describing the limiting velocity of random walks in random environments, and models a large class of disordered systems quite naturally. To this recognition, we have the following definition.
    %definition: Guo's time-correlated condition
    \begin{definition}\label{def: Guo's time-correlated}
        Let $C$ and $g$ be given and fixed positive constants. A Markovian field $\omega$ with law $(\mathbb{P},\mathscr{F}_\Omega)$ on $\mathbb{N}\times\mathbb{Z}^d$ is said to satisfy the Guo's time-correlated condition $\textbf{(TCG)}_{C,g}$ if for any finite subsets $\Delta\subseteq V\subseteq \mathbb{N}\times\mathbb{Z}^d$ with $d_1(\Delta,V^c)>1$ and $A\subseteq V^c$,
    \[
        \frac{d\mathbb{P}((\omega_{n,z})_{(n,z)\in\Delta}\in\vdot|\eta)}{d\mathbb{P}((\omega_{n,z})_{(n,z)\in\Delta}\in\vdot|\eta^\prime)} \leq \exp\bigg(C\sum_{(m,x)\in\Delta,(k,y)\in A}e^{-g\abs{x-y}_1-g\abs{m-k}}\bigg)
    \]
    simultaneously for all pairs of configurations $\eta,\eta^\prime\in\Omega$ which agree on $V^c\backslash A$, $\mathbb{P}$-a.s.
    \end{definition}
    Intuitively $\textbf{(TCG)}_{C,g}$ is an asymptotic more general assumption. Strictly speaking, the former is not implied by condition $\textbf{(TC)}_{C,g}$, but in asymptotic terms Guo's time-correlated condition is harder to work with. Our \textbf{standing assumption} on the random field $\omega$ is that it is Markovian and satisfies either the time-correlated condition in Definition \ref{def: time-correlated} or the Guo's condition in Definition \ref{def: Guo's time-correlated}.\par
    Let us now define the $\mathbb{Z}^d$-valued polymer process on $\mathbb{N}$ via its finite-dimensional marginals $P^{\beta,\omega}_N$ at each time point $N\in\mathbb{N}$. This polymer measure $P^{\beta,\omega}_N$ is a probability measure on the path space $((\mathbb{Z}^d)^N,\mathcal{P}(\mathbb{Z}^d)^{\otimes N})$ characterized by its Radon--Nikodým derivative with respect to $P^S_0$ in finite time by
    \begin{equation}\label{eqn: Polymer measure, R-N derivative}
        P^{\beta,\omega}_N(dS) = \frac{1}{Z^{\beta,\omega}_N} \exp\big( \beta\sum_{n=1}^N \omega_{n,S_n} \big) P^S_0(dS),\qquad\forall~S\in\mathcal{S}_N,
    \end{equation}
    where $\mathcal{S}_N$ denotes all possible trajectories of $(S_n)_{n=0}^N$. The normalizing factor $Z^{\beta,\omega}_N$ in (\ref{eqn: Polymer measure, R-N derivative}) is called the \textit{partition function} which ensures $P^{\beta,\omega}_N$ a genuine probability measure up to time point $N$. Our first result, via creating an auxiliary random field to separate the entangled information from the time-correlated field $\omega$, deals with the limiting behavior of the moving average process $(N^{-1}\sum_{k=1}^N \omega_{k,S_k})_{N\geq1}$ in the sense of law of large numbers.
    %Theorem: law of large numbers
    \begin{theorem}\label{thm: law of large numbers}
        \normalfont
        Under either the time-correlated condition $\textbf{(TC)}_{C,g}$ or $\textbf{(TCG)}_{C,g}$ and for each fixed $\beta\geq0$, there exists a deterministic limit $\ell\in\mathbb{R}$ such that
        \[
            \lim_{N\to\infty}\frac{1}{N}\sum_{k=1}^N \omega_{k,S_k}=\ell,\qquad\mathbb{P}\otimes P^S_0\text{-a.s.}
        \]
        We call this limiting identity the law of large numbers for the moving average process $(N^{-1}\sum_{k=1}^N \omega_{k,S_k})_{N\geq1}$.
    \end{theorem}   
    Going beyond the law of large numbers, we investigate the statistical mechanics properties of the polymer process (\ref{eqn: Polymer measure, R-N derivative}). The monotonicity and smoothness of free energies has been shown in \cite{Comets/Yoshida} in i.i.d. environments. Our first result, via creating an auxiliary random field, confirms the same regularity in the more delicate time-correlated structure.
    %Theorem: quenched and annealed free energies exist
    \begin{theorem}\label{thm: quenched and annealed free energies exist}
        \normalfont
        Under either the time-correlated condition $\textbf{(TC)}_{C,g}$ or $\textbf{(TCG)}_{C,g}$, the following limits exist $\mathbb{P}$-a.s. and are of constant value. Namely,
        \[
            \rho(\beta)\coloneqq\lim_{N\to\infty}\frac{1}{N}\log Z^{\beta,\omega}_N\qquad\text{and}\qquad \lambda(\beta)\coloneqq\lim_{N\to\infty}\frac{1}{N}\log\mathbb{E}[Z^{\beta,\omega}_N],\quad\mathbb{P}\text{-a.s.}
        \]
        In particular, the annealed free energy $\beta\mapsto\lambda(\beta)$ is differentiable on $[0,\infty)$. Furthermore, $\rho(\beta)\leq\lambda(\beta)$ for any $\beta\geq0$. And as a function, $\beta\mapsto\rho(\vdot)-\lambda(\vdot)$ is continuous and non-increasing.
    \end{theorem}
    The first limit above is called the \textit{quenched free energy} and the second limit is called the \textit{annealed free energy}. The fact that $\rho(\beta)$ exists exploits the exponential time-correlation and is partly due to the finite exponential moment (\ref{eqn: finit exponential moment of omega}) of the time-correlated field $\omega$, in light of \cite{Rassoul-Agha/Seppalainen/Yilmaz}. Also remark that unless in the i.i.d. underlying field, the annealed free energy $\lambda(\beta)$ is not necessarily equal to $\log\mathbb{E}[e^{\beta\omega_{n,z}}]$ for any $(n,z)\in\mathbb{N}\times\mathbb{Z}^d$ due to the time-correlation.\par
        In light of Theorem \ref{thm: quenched and annealed free energies exist}, we know $\rho(\beta)\leq\lambda(\beta)$ for any $\beta\geq0$. However, whether equality is achieved is yet to be determined. As a function, $\beta\mapsto\rho(\vdot)-\lambda(\vdot)$ is continuous and non-increasing. And hence in principle, there exists a critical value $\beta_*\in[0,\infty]$ such that $\rho(\beta)=\lambda(\beta)$ when $\beta<\beta_*$, and $\rho(\beta)<\lambda(\beta)$ when $\beta>\beta_*$. Following the standard terminologies \cite{Comets}, we say the polymer is in delocalized phase, or \textit{delocaization} if $\beta<\beta_*$; and it is in localized phase, or \textit{localization} if  $\beta>\beta_*$. The aim of this paper is to locate the value of $\beta_*$ when the underlying walk $(S_n)_{n\geq0}$ is in transient dimensions and travels on a time-correlated random field.
    %theorem: main results for beta*>0
    \begin{theorem}\label{thm: main results for beta*>0}
        \normalfont
        Under either the time-correlated condition $\textbf{(TC)}_{C,g}$ or $\textbf{(TCG)}_{C,g}$, there exists some small $\beta>0$ such that $\rho(\beta)=\lambda(\beta)$, i.e.~we always have $\beta_*>0$. In particular, if
        \[
            \lim_{\beta\nearrow\infty}\Lambda(\beta)=\lim_{\beta\nearrow\infty}\log \kappa_1(\beta)\kappa_2(\beta) + \log\mathbb{E}\big[e^{2\beta\omega_{1,0}}\big] + 2\log\mathbb{E}\big[e^{-\beta\omega_{1,0}}\big] < K
        \]
        for some constant $K>0$ determined in Appendix \ref{appendix b}, then $\beta_*=\infty$. Here $\kappa_1,\kappa_2:[0,\infty)\to\mathbb{R}$ with $\kappa_1(0)=\kappa_2(0)=1$ are continuous in $\beta$ induced by the time-correlation and are specified in Appendix \ref{lem: function Lambda(beta)}.
    \end{theorem}
    In Appendix \ref{lem: annealed free energy is differentiable}, we have shown that the limiting annealed free energy $\beta\mapsto\lambda(\beta)$ is Gâteaux differentiable and monotonic on $[0,\infty)$. And we can therefore properly define its differentials. Remark that the notion of derivatives of Gâteaux \cite{Gateaux} is the same as that of Fréchet on the line \cite{Zorn1}, so here $\lambda^\prime(\beta)$ is the Fréchet differential of annealed free energy as well.

    %theorem: main results for beta*<infinity
    \begin{theorem}\label{thm: main results for beta*<infinity}
        \normalfont
        Under either the time-correlated condition $\textbf{(TC)}_{C,g}$ or $\textbf{(TCG)}_{C,g}$, whenever 
        \[
            \beta\frac{d\lambda}{d\beta}(\beta)-\lambda(\beta)>-K(S)\sum_{x\in\mathbb{Z}^d}P^S_0(S_1=x)\log P^S_0(S_1=x)
        \]
        for some constant $K(S)>0$ depending only on the random walk $(S_n)_{n\geq1}$, then $\rho(\beta)<\lambda(\beta)$. In particular, with $\Bbbk\coloneqq\esssup\omega_{1,0}$, if
        \[
            \log\frac{1}{\mathbb{P}(\omega_{1,0}=\Bbbk)}\geq K^\prime-K(S)\sum_{x\in\mathbb{Z}^d}P^S_0(S_1=x)\log P^S_0(S_1=x),
        \]
        then $0<\beta_*<\infty$, where $K^\prime$ is an absolute constant depending only on the time-correlation intensity $(C,g)$. 
    \end{theorem}

    A particular nontrivial instance of time-correlated $\omega$ which satisfies Theorem \ref{thm: main results for beta*<infinity} is described below via Gaussians. Take $(\varphi_{\Vec{x}})_{\Vec{x}\in\mathbb{N}\times\mathbb{Z}^d}$ to be a centered Gaussian field on $\mathbb{N}\times\mathbb{Z}^d$ with covariance  $\langle\varphi_{\Vec{x}},\varphi_{\Vec{y}}\rangle = \text{exp}(-\|\Vec{x}-\Vec{y}|_{\ell^1(\mathbb{N}\times\mathbb{Z}^d)}) G(\Vec{x},\Vec{y})$ where $G(\vdot)$ denotes the Green function of a Gaussian free field (GFF) \cite{Werner/Powell} on $\mathbb{Z}_+\times\mathbb{Z}^d$. The existence of such $(\varphi_{\Vec{x}})_{\Vec{x}\in\mathbb{N}\times\mathbb{Z}^d}$ follows similarly from the existence of GFF, and could be found in \cite{Werner/Powell}. Obviously, if we take such $(\varphi_{\Vec{x}})_{\Vec{x}\in\mathbb{N}\times\mathbb{Z}^d}$ to be the underlying time-correlated field, then $0<\beta_*<\infty$.\par    
    One should remark that when the underlying field $\omega$ is i.i.d. in space and time, then the constant $K^\prime$ vanishes by our subsequent derivation. We leave it for future project to sharpen the constants $K(S)$ and $K^\prime$ in the above entropy-type criteria. Because $\rho(0)=\lambda(0)=0$ when $\beta=0$. Therefore, without loss of any generality, we assume $\beta>0$.    
%-----------------------------------------
%-----------------------------------------
\section{Two auxiliary fields}\label{sec: auxiliary fields}
    Define $\mathcal{W}\coloneqq\{-1,0,1\}$ and define the product probability measure $Q$ on $\epsilon=(\epsilon_1,\epsilon_2,\ldots)\in(\mathcal{W})^{\mathbb{N}}$ by $Q(\epsilon_1=1)=Q(\epsilon_1=-1)=\frac{1}{4}$, $Q(\epsilon_1=0)=\frac{1}{2}$. We also introduce two auxiliary random fields $\eta=(\eta_{n,z}: \,n\in\mathbb{N},z\in\mathbb{Z}^d)$ and $\xi(l)=(\xi^l_{n,z}:\,n\in\mathbb{N},z\in\mathbb{Z}^d)$ for each $l\geq1$ on $\mathbb{N}\times\mathbb{Z}^d$ by
    \begin{subequations}
        \begin{equation}\label{eqn: auxiliary field eta}
            \eta_{n,z} = 0\vdot\mathbbm{1}_{\{\epsilon_n=\pm1\}} + 2\omega_{n,z} \mathbbm{1}_{\{\epsilon_n=0\}},
        \end{equation}
        and, for each index $l\geq1$,
        \begin{equation}\label{eqn: auxiliary field xi}
            \xi^l_{n,z} = -\beta l\mathbbm{1}_{\{\epsilon_n=\pm1\}}+\log(2e^{\beta\omega_{n,z}}-e^{-\beta l})\mathbbm{1}_{\{\epsilon_n=0\}}
        \end{equation}
    \end{subequations}
    for each $n\in\mathbb{N}$ and $z\in\mathbb{Z}^d$. It is not hard to see that $ E_Q[\eta_{n,z}]=\omega_{n,z} $ for all $(n,z)\in\mathbb{N}\times\mathbb{Z}^d$ and $E_Q[Z^{1,\xi(l)}_N]=Z^{\beta,\omega(l)}_N$ for each time point $N\in\mathbb{N}$ with $\omega(l)=(\max\{\omega_{n,z},-l\}:\,n\in\mathbb{N},z\in\mathbb{Z}^d)$, and that the new indexed auxiliary field $\xi(l)$ satisfies (\ref{eqn: finit exponential moment of omega}) as well.\par
    Following this convention, we denote $\overline{P}^S_0\coloneqq Q\otimes P^S_0$ and subsequently $\overline{E}^S_0\coloneqq E_{Q\otimes P^S_0}$. Let us now select a random time sequence $(\tau^{(L)}_n)_{n\geq0}$ given by $\tau^{(L)}_0=0$ and
    \[
        \tau^{(L)}_n\coloneqq\inf\big\{j\geq\tau^{(L)}_{n-1}+L:\,(\epsilon_{j-L},\ldots,\epsilon_{j-1})=+1,\ldots,1,\,\epsilon_j=-1,0\big\},\qquad\forall~n\geq1.
    \]
    To quantify the correlated information from the random field, we therefore also define the following random field $\sigma$-algebras on $(\mathbb{N}\times\mathbb{Z}^d)\times(\mathcal{W})^{\mathbb{N}}$ by $\mathscr{G}_0\coloneqq\sigma(\omega_{k,z}:\,k\leq-L,\,z\in\mathbb{Z}^d)$, and 
    \[
        \mathscr{G}_n\coloneqq\sigma\big(\tau^{(L)}_1,\ldots,\tau^{(L)}_n,\,\omega_{k,z}:\,k\leq\tau^{(L)}_n-L,\,z\in\mathbb{Z}^d,\,\epsilon_i:\,i=1,\ldots,\tau^{(L)}_n\big),\qquad\forall~n\geq1.
    \]
    We also define the specific field $\sigma$-algebras $\mathscr{F}^{(L)}_n\coloneqq\sigma(\omega_{k,z}:\,k\leq n-L,\,z\in\mathbb{Z}^d,\,\epsilon_i:\,i=1,\ldots,n)$.\par
    The spirit of this type of auxiliary structure was designed to resolve the correlated and mixing random medium in Comets/Zeitouni \cite{Comets/Zeitouni} in the context of random walk in random environment, and later adopted by Guerra Aguilar \cite{Guerra Aguilar} and myself \cite{Chen} to deal with the central limit theorem and large deviation principles for the same model, respectively. To the best of our knowledge, this techniques has not been introduced in the study of directed polymers before.\par
    Define $W^{\beta,\omega}_N\coloneqq\frac{Z^{\beta,\omega}_N}{\mathbb{E}[Z^{\beta,\omega}_N]}$. Choose some sufficiently small angle $\zeta_0>0$ such that the space-time cones
    \[
        C(k,x,\gamma,\zeta)\coloneqq\big\{(n,z)\in\mathbb{N}\times\mathbb{Z}^d:\,\langle \Vec{z}-\Vec{x},\gamma\rangle\geq\zeta\abs{\gamma}_2\abs{\Vec{z}-\Vec{x}}_2\big\},\quad\forall~(k,x)\in\mathbb{N}\times\mathbb{Z}^d,\;\;0<\zeta<\zeta_0,
    \]
    admits $S_{n+1}\in C(n,S_n,\gamma,\zeta)$ for any $n\in\mathbb{N}$, $P^S_0$-a.s. Here we write $\Vec{z}\coloneqq(n,z)$, $\Vec{x}\coloneqq(k,x)$, and $\gamma\in\mathbb{S}^d$ is a $(d+1)$-dimensional directional vector. It is not hard to see that $\tau^{(L)}_1<\cdots<\tau^{(L)}_n<\infty$ for any $n\geq1$. In fact, we have the following lemmas.
    
    %lemma: tau lower bound
    \begin{lemma}\label{lem: tau lower bound}
        \normalfont
        For any real $p\geq1$, there exists some constants $c_p,c^\prime_p>0$ independent of $L$ such that
        \[
            c_p\leq E_Q\big[(\Bar{\tau}^{(L)}_1)^p\big]^{1/p}\leq c^\prime_p
        \]
        for all $L$, where we let $\Bar{\tau}^{(L)}_1\coloneqq4^{-L}\tau^{(L)}_1$.
    \end{lemma}
    \begin{proof}
        We divide the proof of this lemma into several steps.\par\noindent
        \textbf{Step I.} Lower-bound.\par\noindent
        Define a Markov Chain $(U_n)_{n\geq1}$ on state space $\{0,1,\ldots,L\}$, with $U_0=0$ and 
        \[
            U_n=\max\big\{k\geq1:\,(\epsilon_{n-k+1},\ldots,\epsilon_n)=(1,\ldots,1)\big\}\vee0.
        \]
        Then, it is not hard to see that $\tau^{(L)}_1\geq\min\{n\geq1:\,U_n=L\}$, see \cite[p. 892]{Comets/Zeitouni}. Consider the successive times when $U_n=1$, then $\tau^{(L)}_1$ can be lower-bounded by a sum of a $\text{Geometric}(4^{-L+1})$ number of independent random variables which are bounded below by $1$. Thus,
        \[
            \varliminf_{L\to\infty} Q\big(\tau^{(L)}_1\geq\theta4^{-L}\big)\geq g(\theta)\qquad\text{for some}\quad g(\theta)\xrightarrow[]{\theta\to0}1.
        \]
        Therefore,
        \[
            E_Q\big[\Bar{\tau}^{(L)}_1\big]\geq 4^{-L}E_Q\big[\tau^{(L)}_1\mathbbm{1}_{\{ \tau^{(L)}_1\geq\theta4^{-L} \}}\big]\geq \theta\big(1-Q(\tau^{(L)}_1<\theta4^{-L})\big)\geq c>0,
        \]
        provided that we choose $\theta$ sufficiently small such that $\varlimsup_{L\to\infty} Q(\tau^{(L)}_1<\theta4^{-L})<\frac{1}{2}$.\par\noindent
        \textbf{Step II.} Upper-bound.\par\noindent
        We will actually prove the exponential moment of $\Bar{\tau}^{(L)}_1$ is finite, which is nevertheless stronger then the claim. Define the events for each $n\geq1$ by
        \[
        A_n\coloneqq\big\{\epsilon\in(\mathcal{W})^{\mathbb{N}}:\,(\epsilon_{n-L},\ldots,\epsilon_{n-1})=1,\ldots,1,\,\epsilon_n=-1,0\big\}
        \]
        and 
        \[
            B_n\coloneqq\big\{\epsilon\in(\mathcal{W})^{\mathbb{N}}:\,(\epsilon_{j-L},\ldots,\epsilon_{j-1},\epsilon_j)\neq1,\ldots,1,-1\,\text{or}\,0,\,\forall~L\leq j\leq n-L-1\big\}.
        \]
        Hence,
        \begin{equation*}\begin{aligned}
            E_Q\big[e^{\Bar{\tau}^{(L)}_1}\big]/3 &\leq \sum_{n=1}^{2L} E_Q\big[e^{4^{-L}n},\,A_n\big] +\sum_{n=2L+1}^\infty E_Q\big[e^{4^{-L}n},\,A_n,\,B_n\big]\\ 
            &\leq (L^2-1)4^{-(L+1)} e^{L4^{-(L-1)}} + 4^{-(L+1)} \sum_{k=L^2}^\infty E_Q\big[e^{4^{-L}n},\,B_n\big].
        \end{aligned}\end{equation*}
        By \cite[Lemma 6.6]{Guerra Aguilar/Ramirez}, $Q(B_n)\leq (1-cL^2 4^{-L})^{ \lfloor n/L^2\rfloor }$ for each $n\geq L^2$. Hence,
        \[
            E_Q\big[e^{\Bar{\tau}^{(L)}_1}\big] \leq K + \tfrac{3}{2}L^2 4^{-L} \sum_{k=1}^\infty \big(e^{L^2 4^{-L}}(1-cL^2 4^{-L})\big)^k\leq K+ \tfrac{3 K^2 4^{-L}/2}{e^{-L^2 4^{-L}}- (1-cL^2 4^{-L}) } \leq K+K^\prime<\infty,
        \]
        where $K,K^\prime$ are absolute constants independent of $L$. And the assertion is verified.      
    \end{proof}
    %lemma: tau mixing inequality
    \begin{lemma}\label{lem: tau mixing inequality}
        \normalfont
        Let $f:(\mathbb{R})^{\mathbb{N}}\to\mathbb{R}$ be a bounded Borel measurable function. For any $n\in\mathbb{N}$, we abbreviate both the finite-time process $(\sum_{k=\tau^{(L)}_{n-1}+1}^i \eta_{k,S_k})_{\tau^{(L)}_{n-1}\leq i\leq\tau^{(L)}_n}$ and the process $(\sum_{k=\tau^{(L)}_{n-1}+1}^i \xi^l_{k,S_k})_{\tau^{(L)}_{n-1}\leq i\leq\tau^{(L)}_n}$ by $S^{\tau,n}_{\vdot}$. Then, under either $(\textbf{TC})_{C,g}$ or $(\textbf{TCG})_{C,g}$ and for each $t\in(0,1)$ there exists $L_0 = L_0(C, g, d) \in \mathbb{N}$ so that $\mathbb{P}\otimes P^S_0\otimes Q$-a.s. for all $L\geq L_0$,
        \begin{equation*}\begin{aligned}                                        
            \exp \big(-e^{-gtL}\big) \mathbb{E}\overline{E}^S_0 \big[f(S^{\tau,1}_{\vdot})\big] \leq \mathbb{E}\overline{E}^S_0 \big[f(S^{\tau,n}_{\vdot}) \big|\mathscr{G}_{n-1}\big]\leq \exp\big(e^{-gtL}\big)\mathbb{E}\overline{E}^S_0\big[f(S^{\tau,1}_{\vdot})\big].
        \end{aligned}\end{equation*}
    \end{lemma}
    \begin{proof}
        Take bounded and $\mathscr{G}_{n-1}$-measurable $h:(\mathbb{R})^{\mathbb{N}\times\mathbb{Z}^d}\times(\mathcal{W})^{\mathbb{N}}\to\mathbb{R}$, where $\mathcal{W}\coloneqq\{-1,0,1\}$. Then,
        \[
            \mathbb{E}\overline{E}^S_0 \big[f(S^{\tau,n}_{\vdot})h\big] = \sum_{k\in\mathbb{N}} \mathbb{E}\overline{E}^S_0 \big[f(S^{\tau,n}_{\vdot}) h_k,\,\tau^{(L)}_{n-1}=k\big],
        \]
        because on the event $\{\tau^{(L)}_{n-1}=k\}$, we can find a bounded function $h_{k}$ which is $\mathscr{F}^{(L)}_{k}$-measurable and coincides with $h$ on this event. Then, we observe that
        \begin{equation}\label{eqn: decompose conditional expectation}
            \mathbb{E}\overline{E}^S_0\big[f(S^{\tau,n}_{\vdot}) h\big] = \sum_{k\in\mathbb{N}} \mathbb{E}\overline{E}^S_0 \big[ h_k,\, \tau^{(L)}_{n-1}=k,\, \mathbb{E}\overline{E}^S_0 [f(S^{\tau,1}_{k+\vdot})|\mathscr{F}^{(L)}_{k}]\big],
        \end{equation}
        Consider the space-time hyperplane $\mathbb{H}_{L,x,s}$ and the cone-like region $\mathbb{C}_{x,s}$ defined respectively by
        \[
            \mathbb{H}_{L,k,x}\coloneqq\big\{(m,z)\in\mathbb{N}\times\mathbb{Z}^d:\, m-k\leq-L \big\} \quad\text{and}\quad\mathbb{C}_{k,x}\coloneqq C(k,x,\hat{\gamma},\zeta_0),
        \]
        where $\hat{\zeta}$ is the unit vector in the time-direction. In terms of time-correlation $\textbf{(TC)}_{C,g}$, we first estimate the series
        \begin{equation}\label{eqn: estimate series1, SM}
            \sum_{\Vec{y}\in\partial^r\mathbb{H}_{L,k,x}} \sum_{\Vec{z}\in\partial^r\mathbb{C}_{k,x}}\exp\big(-g\abs{\Vec{y}-\Vec{z}}_1\big).
        \end{equation}
        And in terms of Guo's time-correlation $\textbf{(TCG)}_{C,g}$, we second estimate the series
        \begin{equation}\label{eqn: estimate series2, SMG}
            \sum_{\Vec{y}\in\mathbb{H}_{L,k,x}} \sum_{\Vec{z}\in\mathbb{C}_{k,x}}\exp\big(-g\abs{\Vec{y}-\Vec{z}}_1\big).
        \end{equation}
        Notice that with $L$ sufficiently large, both series converge because $\mathbb{C}_{k,x,s}$ is cone-like. Indeed, choose $\hat{t}\in(t,1)$ and consider (\ref{eqn: estimate series1, SM}). We take $L$ large enough such that $L>(1-\hat{t})^{-1}2r$, and thus $L-2r>\hat{t}L$. Setting
        \[
            \mathbb{K}_{L,k,x,n}\coloneqq\big\{(\Vec{y},\Vec{z}):\,\Vec{y}\in\partial^r\mathbb{H}_{L,k,x},\,\Vec{z}\in\partial^r\mathbb{C}_{k,x},\,\hat{t}L+n\leq\abs{\Vec{y}-\Vec{z}}_1<\hat{t}L+n+1\big\}.
        \]
        Whence we have
        \[
            \sum_{\Vec{y}\in\partial^r\mathbb{H}_{L,k,x}} \sum_{\Vec{z}\in\partial^r\mathbb{C}_{k,x}}\exp\big(-g\abs{y-z}_1\big) \leq \sum_{n\geq0}\sum_{(\Vec{y},\Vec{z})\in\mathbb{K}_{L,k,x,n}} e^{-g\abs{\Vec{y}-\Vec{z}}_1} \leq \sum_{n\geq0} \abs{\mathbb{K}_{L,k,x,n}}e^{-g(\hat{t}L+n)}.
        \]
        Since $\abs{\mathbb{K}_{L,k,x,n}}\leq Cr^2(n+1)^{2(d-1)}$, we have
        \begin{equation}
            \sum_{\Vec{y}\in\partial^r\mathbb{H}_{L,k,x}} \sum_{\Vec{z}\in\partial^r\mathbb{C}_{k,x}}\exp\big(-g\abs{\Vec{y}-\Vec{z}}_1\big) \leq \sum_{n\geq0} (n+1)^{2(d-1)} e^{-gn} \leq e^{-g\Tilde{t}L},\quad\text{with}~\Tilde{t}\in(t,\hat{t}),
        \end{equation}
        yielding an estimate to (\ref{eqn: estimate series1, SM}). Alternatively, we consider (\ref{eqn: estimate series2, SMG}). Performing a very similar argument, we can find some $\Tilde{t}\in(t,1)$ such that
        \begin{equation}
            \sum_{\Vec{y}\in\mathbb{H}_{L,k,x}} \sum_{\Vec{z}\in\mathbb{C}_{k,x}}\exp\big(-g\abs{\Vec{y}-\Vec{z}}_1\big) \leq e^{-g\Tilde{t}L},
        \end{equation}
        which yields an estimate to (\ref{eqn: estimate series2, SMG}). Hence, in terms of $\textbf{(TC)}_{C,g}$, for a given finite path $x_{\vdot}=(x_i)_{0\leq i\leq m}$ in $ \mathbb{C}_{k_0,x_0}$ starting from $(k_0,x_0)$, we have that uniformly on $m$ and with sufficiently large $L$, 
        \begin{equation}\label{eqn: abc}                    
            \sum_{\Vec{y}\in\partial^r\mathbb{H}_{L,k_0,x_0}} \sum_{\Vec{z}\in\partial^r\mathbb{G}_{k_0,x_0}}\exp\big(-g\abs{\Vec{y}-(\Vec{z}-(k_0,x_0))}_1\big) \leq e^{-g\Tilde{t}L},
        \end{equation}
        where we denote $\mathbb{G}_{k_0,x_0}\coloneqq\{\Vec{y}\in\mathbb{N}\times\mathbb{Z}^d:\,\Vec{y}=x_i~\text{for some}~0\leq i\leq m\}$. Likewise, in terms of $\textbf{(TCG)}_{C,g}$, we can find sufficiently large $L$ so that uniformly on $m$,
        \begin{equation}\label{eqn: def}
            \sum_{\Vec{y}\in\mathbb{H}_{L,k_0,x_0}} \sum_{0\leq i\leq m}\exp\big(-g\abs{\Vec{y}-(x_i-x_0)}_1\big) \leq e^{-g\Tilde{t}L}.
        \end{equation}
        From the last term in (\ref{eqn: decompose conditional expectation}), we have
        \begin{equation*}
            \mathbb{E}\overline{E}^S_0 \big[f(S^{\tau,1}_{k+\vdot})\big|\mathscr{F}^{(L)}_{k}\big] = \mathbb{E}\overline{E}^S_0\big[f(S^{\tau,1}_{k+\vdot}) \big|\mathscr{F}_{\mathbb{H}_{L,\hat{\gamma}k}}\big],
        \end{equation*}
        because $Q$ is a product probability measure on $(\mathcal{W})^{\mathbb{N}}$. We further denote by 
        \[
            \mathbb{H}^{(n)}_{L,k,x}\coloneqq\mathbb{H}_{L,k,x}\cup\{\Vec{z}\in\mathbb{N}\times\mathbb{Z}^d:\,\abs{\Vec{z}}_1\geq n\},\quad\forall~(k,x)\in\mathbb{N}\times\mathbb{Z}^d.
        \]
        Either the condition $\textbf{(TC)}_{C,g}$ [Definition \ref{def: time-correlated}] or the condition $\textbf{(TCG)}_{C,g}$ [Definition \ref{def: Guo's time-correlated}] together with (\ref{eqn: abc}) implies that
        \[
            \exp\big(-Ce^{-g\Tilde{t}L}\big) \leq \frac{ \mathbb{E}\overline{E}^S_0\big[f(S^{\tau,1}_{k+\vdot}),\,\gamma = (S_{k+i}-S_k)_{0\leq i\leq \tau^{(L)}_1} \big|\mathscr{F}_{\mathbb{H}^{(n)}_{L,\hat{\gamma}k}}\big] }{\mathbb{E} \overline{E}^S_0\big[f(S^{\tau,1}_{k+\vdot}),\,\gamma = (S_{k+i}-S_k)_{0\leq i\leq \tau^{(L)}_1} \big]}\leq \exp\big(Ce^{-g\Tilde{t}L}\big)
        \]
        uniformly for bounded $f$ and for any finite path $\gamma$ satisfying (\ref{eqn: abc}) [or respectively (\ref{eqn: def})]. Taking some suitable $\Bar{t}\in(t,\Tilde{t})$ and letting $n\to\infty$, we have
        \begin{equation}\label{eqn: ghi}                   
            \exp\big(-e^{-g\Bar{t}L}\big) \leq \frac{\sum_\gamma \mathbb{E}\overline{E}^S_0\big[f(S^{\tau,1}_{k+\vdot}),\,\gamma = (S_{k+i}-S_k)_{0\leq i\leq \tau^{(L)}_1} \big|\mathscr{F}_{\mathbb{H}^{(n)}_{L,\hat{\gamma}k}}\big] }{\sum_\gamma \mathbb{E} \overline{E}^S_0\big[f(S^{\tau,1}_{k+\vdot}),\,\gamma = (S_{k+i}-S_k)_{0\leq i\leq \tau^{(L)}_1} \big]} \leq \exp\big(e^{-g\Bar{t}L}\big).
        \end{equation}
        Taking (\ref{eqn: ghi}) into (\ref{eqn: decompose conditional expectation}), letting $L_0(t)$ be sufficiently large and $L\geq L_0$, we verify the assertion.       
    \end{proof}
    
    %lemma: kappa-bound for a cone
    \begin{lemma}\label{lem: kappa-bound for a cone}
        \normalfont
        For any $(n,z)\in\mathbb{N}\times\mathbb{Z}^d$ and $\Delta\subseteq\mathbb{N}\times\mathbb{Z}^d$ such that $d_1(\Vec{z},\Delta)>1$, $\Delta\subseteq C(n,z,\gamma,\zeta)$ for some $\gamma\in\mathbb{S}^d$ and positive $\zeta<\zeta_0$, then there exists sufficiently large $\kappa=\kappa(C,g)>0$ such that
    \begin{equation}\label{eqn: kappa-correlation condition}
        \mathbb{E}\big[e^{\beta\omega_{n,z}+\sum_{\nu\in I} \beta\omega_{\nu} }\big]\leq \kappa \vdot \mathbb{E}\big[e^{\beta\omega_{n,z}}\big]\mathbb{E}\big[e^{\sum_{\nu\in I} \beta\omega_{\nu} }\big],\qquad\forall~\beta\geq0\quad\text{and}\quad\text{finite}~I\subseteq\Delta.
    \end{equation}
    \end{lemma}
    \begin{proof}
        It is obvious that 
        \[
            \mathbb{E}\big[e^{\beta\omega_{n,z}+\sum_{\nu\in I} \beta\omega_{\nu} }\big] = \mathbb{E}\big[ \mathbb{E}[e^{\beta\omega_{n,z}}|\mathscr{F}_I]\vdot e^{\sum_{\nu\in I} \beta\omega_{\nu}} \big].
        \]
        Under $(\mathbf{TC})_{C,g}$, since $d_1(\Vec{z},I)>0$, we have
        \[
            \mathbb{E}\big[e^{\beta\omega_{n,z}}\big|\mathscr{F}_I\big] \leq \exp\bigg(C\sum_{(m,x)\sim\Vec{z}} \sum_{(k,y)\in\partial C(n,z,\gamma,\zeta)} e^{-g\abs{x-y}_1-g\abs{m-k}} \bigg) \mathbb{E}\big[e^{\beta\omega_{n,z}}\big] \leq K(C,g) \mathbb{E}\big[e^{\beta\omega_{n,z}}\big],
        \]
        for some $K(C,g)<\infty$. Notice that the positive angle $\zeta$ of the cone ensures the summability of the series. Similarly, under $(\textbf{TCG})_{C,g}$ we have
        \[
            \mathbb{E}\big[e^{\beta\omega_{n,z}}\big|\mathscr{F}_I\big] \leq \exp\bigg(C \sum_{(k,y)\in C(n,z,\gamma,\zeta)} e^{-g\abs{z-y}_1-g\abs{n-k}} \bigg) \mathbb{E}\big[e^{\beta\omega_{n,z}}\big] \leq K^\prime(C,g) \mathbb{E}\big[e^{\beta\omega_{n,z}}\big],
        \]
        for some other $K^\prime(C,g)<\infty$. And then the assertion follows.
    \end{proof}

%--------------------------------------
%--------------------------------------
\section{Law of large numbers}\label{appendix a}
    Recall the splitting representation: If $\Bar{X},\Tilde{X}$ are random variables of law $\Bar{P},\Tilde{P}$ with $\normx{\Bar{P}-\Tilde{P}}_{FV}\leq a<1$, then on an enlarged probability space there exists independent $Y,\delta,Z,\Tilde{Z}$, where $\Delta\sim\text{Bernoulli}(a)$ and
    \begin{equation}\label{eqn: A eqn1}            
        \Bar{X}=(1-\Delta)Y+\Delta Z\qquad\text{and}\qquad\Tilde{X}=(1-\Delta)Y+\Delta\Tilde{Z}.
    \end{equation}
    For a proof, see \cite[Appendix A.1]{Barbour/Holst/Janson}. Although, the exact form of $Y,Z$ are complicated, we nevertheless have the estimates $\Bar{X}=(1-\Delta)\Tilde{X}+\Delta Z$, $|\Delta Z|\leq|\Bar{X}|$ and $|\Delta\Tilde{Z}|\leq|\Tilde{X}|$. Note that by recursive conditioning, this result extends to random sequences.
    %lemma: A1
    \begin{lemma}\label{lem: A1}
        \normalfont
        Given a random sequence $(X_i)_{i\geq1}$ with law $P$ such that for some probability measure $Q$,
        \[
            \normy{P(\Bar{X}\in\vdot|\Bar{X}_j,\,j<i)-Q }_{FV}\leq a <1.
        \]
        Then there exists an i.i.d sequence $(\Tilde{X}_i,\Delta_i)_{i\geq1}$ such that $\Tilde{X}_1\sim Q$ and $\Delta_1\sim\text{Bernoulli}(a)$ on an enlarged probability space, as well as a sequence $(Z_i)_{i\geq1}$ with $\Delta_i$ independent of $\sigma(\Tilde{X}_j,\Delta_j:\,j<i)\vee\sigma(Z_i)$ and
        \[
            \Tilde{X}_i=(1-\Delta_i)\Tilde{X}_i + \Delta_iZ_i,\qquad\forall~i\geq1.
        \]
    \end{lemma}
    \begin{proof}
        Suppose the assertion has been verified for $i-1$. Apply (\ref{eqn: A eqn1}) with laws $P(\Bar{X}_i\in\vdot|\Bar{X}_j,\,j<i)$ and $Q$, we obtain of $\Delta_i,Z_i,\Tilde{X}_i$ with desired properties. See also \cite[Lemma 2.1]{Berbee} and \cite[Chapter 3]{Thorisson}.
    \end{proof}
    We establish now the law of large numbers for the moving average of the polymer process, which facilitates the overall understanding of the limiting behavior of the directed polymer in high-dimensions.
    %lemma: law of large numbers
    \begin{lemma}\label{lem: law of large numbers}
        \normalfont
        For fixed $\beta>0$, there exists a deterministic limit $\ell\in\mathbb{R}$ such that
        \[
            \lim_{N\to\infty}\frac{1}{N}\sum_{k=1}^N \eta_{k,S_k} = \lim_{N\to\infty}\frac{1}{N}\sum_{k=1}^N \omega_{k,S_k}=\ell,\qquad\mathbb{P}\otimes P^S_0\otimes Q\text{-a.s.}
        \]
    \end{lemma}
    \begin{proof}
        For each $i\geq1$, let $\Bar{X}^{(L)}_i\coloneqq \sum_{k=\tau^{(L)}_{i-1}+1}^{\tau^{(L)}_i} 4^{-L}\eta_{k,S_k}$ and $\Bar{\tau}^{(L)}_i\coloneqq 4^{-L}(\tau^{(L)}_i-\tau^{(L)}_{i-1})$, and let $\mu^{(L)}(\vdot)$ denote the law of $\Bar{X}^{(L)}_1$. By Lemma \ref{lem: tau mixing inequality}, it is easily seen for any $k\geq2$ that
        \[
            \abs{ \mathbb{E}\overline{E}^S_0\big[ \Bar{X}^{(L)}_k\in A\big|\mathscr{G}_{k-1} \big] - \mu^{(L)}(A)  }\leq\psi_L<1,\qquad\forall~\text{Borel}~A\subseteq\mathbb{R},
        \]
        where $\psi_L\coloneqq\exp(e^{-gtL})-1$. Taking supremum over all Borel sets $A\subseteq\mathbb{R}$, we get
        \[
            \normy{ \mathbb{E}\overline{E}^S_0\big[\Bar{X}^{(L)}_k\in\vdot\big|\mathscr{G}_{k-1}\big] - \mu^{(L)}(\vdot) }_{FV} \leq \psi_L,\qquad\forall~k\geq2.
        \]
        Invoking Lemma \ref{lem: A1}, we find the i.i.d. sequence $(\Tilde{X}^{(L)}_i,\Delta^{(L)}_i)_{i\geq1}$ so that $\Tilde{X}^{(L)}_1\sim\mu^{(L)}$, $\Delta^{(L)}_1\sim\text{Bernoulli}(\psi_L)$, along with the other sequence $(Z^{(L)}_i)_{i\geq1}$ satisfying
        \[
            \Bar{X}^{(L)}_i=(1-\Delta^{(L)}_i)\Tilde{X}^{(L)}_i+\Delta^{(L)}_i Z^{(L)}_i.
        \]
        We also write the enlarged $\sigma$-algebra $\Tilde{\mathscr{G}}_i\coloneqq\mathscr{G}_i\vee\sigma(\Tilde{X}^{(L)}_j,\Delta^{(L)}_j:\,j\leq i)$. Notice that we also have $|\Delta^{(L)}_i Z^{(L)}_i|\leq |\Bar{X}^{(L)}_i|$ for each $i\geq1$. Henceforth, by Hölder's inequality and for all real and even $p>1$,
        \begin{equation}\label{eqn: A2 eqn}
            \psi_L\mathbb{E}\overline{E}^S_0\big[(Z^{(L)}_i)^p\big|\Tilde{\mathscr{G}}_{i-1}\big] \leq \mathbb{E}\overline{E}^S_0\big[(\Delta^{(L)}_i Z^{(L)}_i)^p\big|\Tilde{\mathscr{G}}_{i-1}\big] \leq 2K(p) 2^{-pL }\exp(e^{-gtL}) E_Q\big[(\tau^{(L)}_1)^p\big],\;\;\mathbb{P}\otimes P^S_0\otimes Q\text{-a.s.,}
        \end{equation}
        where $K(p)\coloneqq \max\{\mathbb{E}[e^{2p\omega_{1,0}}],\mathbb{E}[e^{-2p\omega_{1,0}}]\}+1<\infty$ and we have used the fact that
        \begin{equation*}\begin{aligned}
            &\mathbb{E}\overline{E}^S_0\big[(\Bar{X}^{(L)}_i)^p\big] \leq \mathbb{E}\overline{E}^S_0\bigg[ \big( \sum_{k=\tau^{(L)}_{i-1}+1}^{\tau^{(L)}_i} e^{\eta_{k,S_k}} \big)^p \bigg]\\ 
            &\quad \leq 2^{-pL} E_Q\bigg[ (\tau^{(L)}_{i}-\tau^{(L)}_{i-1})^{p-1} \sum_{p^\prime=\pm p} \sum_{k=\tau^{(L)}_{i-1}+1}^{\tau^{(L)}_i} \mathbb{E}E^S_0\big[ e^{p^\prime\eta_{k,S_k}}\big| \sigma(\tau^{(L)}_j:\,j\leq i) \big] \bigg]  \leq 2E_Q\big[(\Bar{\tau}^{(L)}_1)^p\big] K(p).
        \end{aligned}\end{equation*}
        Notice that $\varlimsup_{L\to\infty} E_Q[(\Bar{\tau}^{(L)}_1)^p]<\infty$ by Lemma \ref{lem: tau lower bound}. We can express
        \begin{equation}\label{eqn: log need 1}        
            \frac{1}{n}\sum_{i=1}^n\Bar{X}^{(L)}_i=\frac{1}{n}\sum_{i=1}^n \Tilde{X}^{(L)}_i - \frac{1}{n}\sum_{i=1}^n\Delta_i^{(L)}\Tilde{X}^{(L)}_i + \frac{1}{n}\sum_{i=1}^n\Delta^{(L)}_i Z^{(L)}_i,
        \end{equation}
        where first by independence 
        \[
            \frac{1}{n}\sum_{i=1}^n\Tilde{X}^{(L)}_i\xrightarrow[]{n\to\infty}\gamma_L,\qquad\text{where}\quad\gamma_L\coloneqq \mathbb{E}\overline{E}^S_0[\Tilde{X}^{(L)}_1],\qquad\mathbb{P}\otimes P^S_0\otimes Q\text{-a.s.}
        \]
        Meanwhile, for any conjugate $p,q>1$ with $\frac{1}{p}+\frac{1}{q}=1$,
        \[
            \varlimsup_{n\to\infty}\bigg|\frac{1}{n}\sum_{i=1}^n \Delta^{(L)}_i\Tilde{X}^{(L)}_i\bigg|\leq\varlimsup_{n\to\infty} \bigg|\frac{1}{n}\sum_{i=1}^n (\Delta^{(L)}_i)^p \bigg|^{1/p}\vdot \bigg| \frac{1}{n}\sum_{i=1}^n(\Tilde{X}^{(L)}_i)^q \bigg|^{1/q}\leq 2K(q)\psi_L^{1/p} E_Q\big[ (\Bar{\tau}^{(L)}_1)^q \big]^{1/q},
        \]
        $\mathbb{P}\otimes P^S_0\otimes Q$-a.s., where $\eta_L\xrightarrow[]{L}0$ and the last inequality is due to (\ref{eqn: A2 eqn}). Let us define $\Bar{Z}^{(L)}_i\coloneqq \mathbb{E}\overline{E}^S_0[Z^{(L)}_i|\Tilde{\mathscr{G}}_{i-1}]$ for each $i\geq1$. Observe that the process $(M^{(L)}_n)_{n\geq1}$ with each $M^{(L)}_n\coloneqq \sum_{i=1}^n i^{-1}\Delta^{(L)}_i (Z^{(L)}_i-\Bar{Z}^{(L)}_i) $ is a centered $(\Tilde{\mathscr{G}}_n)_{n\geq1}$-martingale. By the Burkholder--Gundy maximal inequality \cite[Eqn. (14.18)] {Zerner/Merkl},
        \begin{equation*}\begin{aligned}
            &\mathbb{E}\overline{E}^S_0\bigg[\big|\sup_{n\geq1} M^{(L)}_n \big|^\gamma\bigg] \leq C(\gamma) \mathbb{E}\overline{E}^S_0\bigg[\sum_{n=1}^\infty \frac{1}{n^2}\big(\Delta^{(L)}_n Z^{(L)}_n - \Delta^{(L)}_n \Bar{Z}^{(L)}_n \big)^2 \bigg]^{\gamma/2}\\ 
            &\qquad\leq C(\gamma)\sum_{n=1}^\infty\frac{1}{n^\gamma} \mathbb{E}\overline{E}^S_0\big[\big(\Delta^{(L)}_n Z^{(L)}_n - \Delta^{(L)}_n \Bar{Z}^{(L)}_n \big)^\gamma\big] \leq C^\prime(\gamma)\qquad\text{with}\quad\gamma\coloneqq p\wedge q.
        \end{aligned}\end{equation*}
        Henceforth, $M^{(L)}_n\xrightarrow[]{n} M^{(L)}_\infty$, $\mathbb{P}\otimes P^S_0\otimes Q$-a.s. with integrable limit $M^{(L)}_\infty$. By Kronecker's lemma \cite[Eqn. (12.7)]{Zerner/Merkl}, it follows that $n^{-1}\sum_{i=1}^n\Delta^{(L)}_i(Z^{(L)}_i -  \Bar{Z}^{(L)}_i)\xrightarrow[]{n}0$, $\mathbb{P}\otimes P^S_0\otimes Q$-a.s. Thus with real and even $q>1$, for any $n\geq1$, by (\ref{eqn: A2 eqn}),
        \[
            |\Bar{Z}^{(L)}_n|\leq \mathbb{E}\overline{E}^S_0\big[ (\Bar{Z}^{(L)}_n)^q\big|\Tilde{\mathscr{G}}_{n-1} \big]^{1/q} \leq K(q)\exp(q^{-1}e^{-gtL})\normx{\Bar{\tau}^{(L)}_1}_{L^q(Q)}\psi_L^{-1/q}\eqqcolon\eta_L\psi_L^{-1/q},.
        \]
        Henceforth by independence, 
        \[
            \varlimsup_{n\to\infty} \bigg|\frac{1}{n}\sum_{i=1}^n\Bar{Z}^{(L)}_i\Delta^{(L)}_i\bigg|\leq \varlimsup_{n\to\infty} \eta_L\psi_L^{-1/q}\frac{1}{n}\sum_{i=1}^n\Delta^{(L)}_i\leq \eta_L\psi_L^{1/p},\qquad\text{where}\quad\frac{1}{p}+\frac{1}{q}=1.
        \]
        Combining all the above estimates, we get
        \begin{equation}\label{eqn: A3 eqn}
            \varlimsup_{n\to\infty}\bigg|\frac{1}{n}\sum_{i=1}^n\Bar{X}^{(L)}_i-\gamma_L\bigg|\leq2\eta_L\psi_L^{1/p}\xrightarrow[]{L\to\infty}0,\qquad\mathbb{P}\otimes P^S_0\otimes Q\text{-a.s.}
        \end{equation}
        On the other hand, it is immediate that
        \[
            \frac{1}{n}\sum_{i=1}^n \Bar{\tau}^{(L)}_i\xrightarrow[]{n\to\infty} \beta_L\coloneqq E_Q\big[\Bar{\tau}^{(L)}_1\big],\qquad Q\text{-a.s.,}
        \]
        by independence. Furthermore, by Lemma \ref{lem: tau lower bound}, $E_Q[\Bar{\tau}^{(L)}_1]\geq c>0$. Therefore, together with (\ref{eqn: A3 eqn}),
        \begin{equation}\label{eqn: A4}
            \varlimsup_{n\to\infty}\bigg| \frac{1}{\tau_n^{(L)}}\sum_{k=1}^{\tau_n^{(L)}}\eta_{k,S_k}-\frac{\gamma_L}{\beta_L} \bigg|\leq\varlimsup_{n\to\infty}\bigg| \frac{n^{-1}\sum_{i=1}^n\Bar{X}^{(L)}_i}{n^{-1}\sum_{i=1}^n\Bar{\tau}^{(L)}_i} - \frac{\gamma_L}{\beta_L} \bigg| \leq C\eta_L\psi_L^{1/p},\qquad\mathbb{P}\otimes P^S_0\otimes Q\text{-a.s.}
        \end{equation}
        Following standard arguments \cite[p. 1864]{Sznitman/Zerner}, we define an increasing sequence $(k_n)_{n\geq1}$ satisfying $\tau^{(L)}_{k_n}\leq n< \tau^{(L)}_{k_n+1}$ for all $n$. Then, we can write
        \[
            \frac{1}{n}\sum_{k=1}^n\eta_{k,S_k} = \frac{k_n}{n}\vdot \frac{1}{k_n}\bigg(\sum_{k=1}^{\tau^{(L)}_{k_n}} \eta_{k,S_k} +\sum_{k=\tau^{(L)}_{k_n+1}}^n \eta_{k,S_k} \bigg).
        \]
        It is already clear that 
        \[
        \frac{k_n}{n}\xrightarrow[]{n\to\infty}4^{-L}\frac{1}{E_Q[\Bar{\tau}^{(L)}_1]},\qquad Q\text{-a.s.}
        \]
        and that
        \[
            \gamma_L-2\eta_L\psi_L^{1/p}\leq 4^{-L}\varliminf_{n\to\infty} \frac{1}{k_n} \sum_{k=1}^{\tau^{(L)}_{k_n}} \eta_{k,S_k} \leq 4^{-L}\varlimsup_{n\to\infty} \frac{1}{k_n} \sum_{k=1}^{\tau^{(L)}_{k_n}} \eta_{k,S_k} \leq \gamma_L+2\eta_L\psi_L^{1/p},\qquad\mathbb{P}\otimes P^S_0\otimes Q\text{-a.s.}
        \]
        Furthermore, let us define $\Bar{\gamma}_L\coloneqq 4^{-L} \mathbb{E}\overline{E}^S_0[\sum_{k=1}^{\tau^{(L)}_1}\abs{\eta_{k,S_k}}]$. Following an almost identical argument to the above computations, we have
        \begin{equation}\label{eqn: log need 2}        
            4^{-L}\varlimsup_{n\to\infty}\frac{1}{k_n}\sum_{k=\tau^{(L)}_{k_n+1}}^n \eta_{k,S_k} \leq 4^{-L} \varlimsup_{n\to\infty}\frac{1}{k_n} \bigg( \sum_{k=1}^{\tau^{(L)}_{k_n+1}}  - \sum_{k=1}^{\tau^{(L)}_{k_n}} \bigg)\abs{\eta_{k,S_k}} \leq \Bar{\gamma}_L + 2\eta_L\psi_L^{1/p} - (\Bar{\gamma}_L-2\eta_L\psi_L^{1/p}).
        \end{equation}
        And similarly, for the lower-bound we have
        \[
            4^{-L}\varliminf_{n\to\infty}\frac{1}{k_n}\sum_{k=\tau^{(L)}_{k_n+1}}^n \omega_{k,S_k} \geq -\Bar{\gamma}_L - 2\eta_L\psi_L^{1/p} + (\Bar{\gamma}_L+2\eta_L\psi_L^{1/p})\xrightarrow[]{L\to\infty}0,\qquad\mathbb{P}\otimes P^S_0\otimes Q\text{-a.s.}
        \]
        Combining the above estimates, we get the desired law of large numbers for $N^{-1}\sum_{k=1}^N\eta_{k,S_k}$ with the limit $\ell=\lim_{L\to\infty}\gamma_L/\beta_L$. On the other hand, for any $n\in\mathbb{N}$ and $M>0$,
        \[
            \mathbb{E}P^S_0\bigg(\frac{1}{n}\sum_{k=1}^n\abs{\omega_{k,S_k}}\geq M\bigg) \leq e^{-Mn}\mathbb{E}E^S_0\big[e^{\sum_{k=1}^n|\omega_{k,S_k}|}\big] \leq K^n e^{-Mn},
        \]
        where $K\coloneqq\kappa(\mathbb{E}[e^{\omega_{1,0}}]+\mathbb{E}[e^{-\omega_{1,0}}])$ and the last inequality is due to Lemma \ref{lem: kappa-bound for a cone}. Now we choose $M_0$ such that $M_0>\log K$. Then, for each $M\geq M_0$ and $N\in\mathbb{N}$,
        \[
            \sum_{n=N}^\infty \mathbb{E}P^S_0\bigg( \frac{1}{n}\sum_{k=1}^n \abs{\omega_{k,S_k}} \geq M \bigg) \leq \sum_{n=N}^\infty K^n e^{-Mn} \leq \frac{K^Ne^{-MN}}{1-Ke^{-M}}<1.
        \]
        Define the event $A_{N,M}\coloneqq\{n^{-1}\sum_{k=1}^n\abs{\eta_{k,S_k}}\leq 2M,\,\forall~n\geq N \}$. Then by construction (\ref{eqn: auxiliary field eta}) we have the estimate $\mathbb{E}\overline{P}^S_0(A_{N,M})\geq1- (1-Ke^{-M})^{-1}K^N e^{-MN}$. Thus, fixing $N_0$, we have $\mathbb{E}\overline{P}^S_0(\cup_{M\geq M_0}A_{N_0,M})=1$. And by the Dominated Convergence Theorem,
        \begin{equation}\label{eqn: dogdogdog}
            \lim_{N\to\infty}\frac{1}{N}\sum_{k=1}^N \omega_{k,S_k} = \lim_{N\to\infty} \frac{1}{N} \sum_{k=1}^N E_Q[\eta_{k,S_k}] = \ell,\qquad\text{conditioned on}~A_{N_0,M},\quad\forall~M\geq M_0.
        \end{equation}
        Henceforth, the above (\ref{eqn: dogdogdog}) actually holds $\mathbb{P}\otimes P^S_0$-a.s., which verifies the assertion.        
    \end{proof}
    \begin{proof}[Proof of Theorem \ref{thm: law of large numbers}]
        Immediately follows from Lemma \ref{lem: law of large numbers}.
    \end{proof}

%--------------------------------------
%--------------------------------------
\section{Limiting free energies}\label{sec: regularity}
    In this section, we show the existence of quenched and annealed free energies at all temperature, along with other regularities indicated in Theorem \ref{thm: quenched and annealed free energies exist}. In particular, the smoothness of the annealed free energy is deferred to Appendix \ref{lem: annealed free energy is differentiable}.
    \begin{proof}[Proof of Theorem \ref{thm: quenched and annealed free energies exist}.] We divide the proof of the theorem into several steps.\par\noindent
    \textbf{Step I.} Existence of free energies.\par\noindent
        The quenched limit $\rho(\beta)$ exists for all $\beta\geq0$ is a consequence of \cite[Theorem 2.3]{Rassoul-Agha/Seppalainen/Yilmaz} in dimension $d\geq3$. For the existence of annealed limit, for any $\beta\geq0$ we first observe that
        \begin{equation}\begin{aligned}\label{eqn: dog1}
            &\sum_{S^\prime\in\mathcal{S}_N} \mathbb{E}\big[e^{\beta\sum_{n=1}^N\omega(n,S^\prime_n)}\big]\prod_{n=1}^N P^S_0\big(S_n-S_{n-1}=S^\prime_n-S^\prime_{n-1}\big)\\
            &\quad\leq  \kappa^N\sum_{S^\prime\in\mathcal{S}_N} P^S_0\big( S_n=S^\prime_n,\,n=1,\ldots,N \big) \prod_{n=1}^N \mathbb{E}\big[e^{\beta\omega(n,S^\prime_n)}\big] \leq \kappa^N C(\beta)^N,\qquad\forall~N\geq1.
        \end{aligned}\end{equation}
        Here the first inequality is due to the $\kappa$-correlation from Lemma \ref{lem: kappa-bound for a cone}, whereas the second inequality above is due to the exponential moment condition (\ref{eqn: finit exponential moment of omega}). Here $C(\beta)=E[e^{\beta\omega_{1,0}}]$ and a similar lower-bound yields
        \[
            \log C(\beta) - \log\kappa \leq \varliminf_{N\to\infty}\frac{1}{N}\log\mathbb{E}\big[Z^{\beta,\omega}_N\big] \leq \varlimsup_{N\to\infty}\frac{1}{N}\log\mathbb{E}\big[Z^{\beta,\omega}_N\big]\leq \log C(\beta)+\log\kappa.
        \]
        Now suppose we could find $\log C(\beta) - \log\kappa\leq a < b \leq \log C(\beta) + \log\kappa$ as well as two subsequences $(N_m)_{m\geq1}$ and $(N_n)_{n\geq1}$ such that
        \[
            a=\lim_{m\to\infty}\frac{1}{N_m} \log \mathbb{E}\big[Z^{\beta,\omega}_{N_m}\big],\qquad b=\lim_{n\to\infty}\frac{1}{N_n} \log \mathbb{E}\big[Z^{\beta,\omega}_{N_n}\big].
        \]
        For any prime $p\in\mathbb{N}$, we find two sequences $(k_m)_{m\geq1}$ and $(k_n)_{n\geq1}$ such that $p k_m\leq N_m<p(k_m+1)$ and $p k_n\leq N_n<p(k_n+1)$ for each $m,n\in\mathbb{N}$. Applying Lemma \ref{lem: kappa-bound for a cone} at each time point of multiple of $p$, we get
        \[
            \lim_{m\to\infty}\frac{1}{N_m} \log \mathbb{E}\big[Z^{\beta,\omega}_{N_m}\big] \leq \frac{1}{p}\log\kappa + \frac{1}{p} \log \mathbb{E}\big[Z^{\beta,\omega}_{p}\big]
        \]
        as well as
        \[
            \lim_{n\to\infty}\frac{1}{N_n} \log \mathbb{E}\big[Z^{\beta,\omega}_{N_n}\big] \geq -\frac{1}{p}\log\kappa + \frac{1}{p} \log \mathbb{E}\big[Z^{\beta,\omega}_{p}\big].
        \]
        Letting $p$ sufficiently large with $p>2(b-a)^{-1}\log\kappa$, we observe a contradiction with the two different limits. Hence, we know the annealed limit $\lambda(\beta)$ exists on $[0,\infty)$. And as a pointwise limit of convex functions, the continuity of $\lambda(\vdot)$ follows automatically from its convexity on $[0,\infty])$.\par\noindent 
        \textbf{Step II.} We prove (\ref{eqn: another way to write rho(beta)}).\par\noindent
        To show $\rho(\beta)\leq\lambda(\beta)$ under the time-correlated field, we first notice that
        \[
            Z^{\beta,\omega}_{N+k}= Z^{\beta,\omega}_N \sum_{x\in\mathbb{Z}^d} P^S_0\big(S_N=x\big) Z^{\beta,\vartheta_{x,N}\omega}_k,
        \]
        where $\vartheta_{x,N}:\omega(\vdot)\mapsto\omega(N+\vdot,x+\vdot)$ denotes the translation map. Hence,
        \[
            \log Z^{\beta,\omega}_{N+k} \geq \log Z^{\beta,\omega}_N +\sum_{x\in\mathbb{Z}^d} P^S_0\big(S_N=x\big) \log Z^{\beta,\vartheta_{x,N}\omega}_k
        \]
        by Jensen's inequality. And then,
        \[
            \mathbb{E}\big[\log Z^{\beta,\omega}_{N+k}\big] \geq \mathbb{E}\big[\log Z^{\beta,\omega}_N\big] + \mathbb{E}\big[\log Z^{\beta,\omega}_k\big]
        \]
        due to the translation invariance of $\omega$. This implies that $\mathbb{E}[\log Z^{\beta,\omega}_N]$ is superaddictive in $N$. By the Fekete's lemma \cite[p. 86]{Rassoul-Agha/Seppalainen}, we know the following limit exists.
        \begin{equation}\label{eqn: another way to write rho(beta)}
            \lim_{N\to\infty}\frac{1}{N}\mathbb{E}\big[\log Z^{\beta,\omega}_N\big] = \sup_{N\in\mathbb{N}} \frac{1}{N}\mathbb{E}\big[\log Z^{\beta,\omega}_N\big] \coloneqq \Tilde{\rho}(\beta),\qquad\forall~\beta\geq0.
        \end{equation}\par\noindent
        \textbf{Step III.} We prove (\ref{eqn: another way to write rho(beta)}) is equal to $\rho(\vdot)$.\par\noindent
        Write the $\sigma$-algebras $\mathscr{K}_j\coloneqq \sigma( \omega_{k,z}:\,k\leq j,\,z\in\mathbb{Z}^d )$. Here, $\mathscr{K}_0$ denotes the trivial $\sigma$-algebra. We can thus express $\log Z^{\beta,\omega}_N-\mathbb{E}[\log Z^{\beta,\omega}_N]$ as a sum of martingale difference,
        \[
            \log Z^{\beta,\omega}_N- \mathbb{E}\big[\log Z^{\beta,\omega}_N\big] = \sum_{j=1}^N V_{j,N},\qquad\text{where}\quad V_{j,N}\coloneqq \mathbb{E}\big[\log Z^{\beta,\omega}_N\big|\mathscr{K}_j\big] - \mathbb{E}\big[\log Z^{\beta,\omega}_N\big|\mathscr{K}_{j-1}\big]
        \]
        for each $j\leq N$. Also, let $\hat{Z}^{\beta,\omega}_{j,N}\coloneqq E^S_0[e^{ \Delta_{j,S_j} + \beta\sum_{k\neq j} \omega_{k,S_k} }]$ where
        \[
            \Delta_{j,x}\coloneqq \tfrac{1}{2}\log\mathbb{E}\big[ e^{2\beta\omega_{j,x}} \big| \mathscr{K}_{j-1} \big] - \tfrac{1}{2}\log \mathbb{E}\big[ e^{2\beta\omega_{j,x}} \big],\qquad\forall~j\leq N.
        \]
        It is then obvious to see that
        \[
            V_{j,N} = \mathbb{E}\big[\text{log} (Z^{\beta,\omega}_N/\hat{Z}^{\beta,\omega}_{j,N})\big|\mathscr{K}_j\big] - \mathbb{E}\big[\text{log}( Z^{\beta,\omega}_N/\hat{Z}^{\beta,\omega}_{j,N})\big|\mathscr{K}_{j-1}\big],\qquad\forall~j=1,\ldots,N.
        \]
        For each $x\in\mathbb{Z}^d$, writing $\alpha_{x,j}\coloneqq P^S_0(e^{\beta\sum_{k\neq j} \omega_{k,S_k} },\,S_j=x)/\hat{Z}^{\beta,\omega}_{j,N}$ and taking $p\in\{\pm1\}$, we then have 
        \[
            \mathbb{E}\big[e^{pV_{j,N}}\big] \leq \mathbb{E}\big[ e^{2p \mathbb{E}[\text{log}(Z^{\beta,\omega}_{N}/\hat{Z}^{\beta,\omega}_{j,N})|\mathscr{K}_{j}] } \big]^{1/2} \vdot \mathbb{E}\big[ e^{-2p \mathbb{E}[\text{log}(Z^{\beta,\omega}_{N}/\hat{Z}^{\beta,\omega}_{j,N})|\mathscr{K}_{j-1}] } \big]^{1/2}\leq C(\beta)\mathbb{E}\big[(Z^{\beta,\omega}_{N}/\hat{Z}^{\beta,\omega}_{j,N})^{2p}\big]^{1/2},
        \]
        where the last inequality is due to Lemma \ref{lem: how to estimate Z and hat Z ratio}. Then,
        \begin{equation*}
            \mathbb{E}\big[e^{pV_{j,N}}\big] \leq C(\beta) \mathbb{E}\big[\big( \sum_{x\in\mathbb{Z}^d} \alpha_{x,j} e^{\beta\omega_{j,x}-\Delta_{j,x}} \big)^{2p}\big]^{1/2} \leq C(\beta) \mathbb{E}\big[ \sum_{x\in\mathbb{Z}^d} \alpha_{x,j} \mathbb{E}\big[ e^{2(\beta\omega_{j,x}-\Delta_{j,x})}\big|\mathscr{K}_{j-1} \big]^{p^2} \big]^{p/2},
        \end{equation*}
        where the last inequality is due to Jensen's inequality. Henceforth, for any $(n,z)\in\mathbb{N}\times\mathbb{Z}^d$,
        \[
            \mathbb{E}\big[e^{pV_{j,N}}\big] \leq C(\beta) \mathbb{E}\big[e^{2\beta\omega_{n,z}}\big]^{p/2}\vdot\mathbb{E}\big[\sum_{x\in\mathbb{Z}^d}\alpha_{x,j}\big]^{p/2}\leq K(\beta)<\infty,\qquad\forall~j\leq N.
        \]
        In light of Lesigne/Volný \cite[Theorem 3.2]{Lesigne/Volny} and that $\exp(\abs{V_{j,N}})\leq\exp(V_{j,N})+\exp(-V_{j,N})$, an application of the large deviation estimate for sum of martingale differences yields
        \begin{equation}\label{eqn: a particular concentration inequality for log Z(beta,N)}
            \mathbb{P}\big(\big|\log Z^{\beta,\omega}_N-\mathbb{E}[\log Z^{\beta,\omega}_N]\big|>\epsilon N\big) \leq e^{-\epsilon^{2/3}N^{1/3}/4},\qquad\forall~\epsilon>0\quad\text{and}\quad N\geq N_0,
        \end{equation}
        for some $N_0=N_0(\beta,\epsilon)\in\mathbb{N}$. From (\ref{eqn: a particular concentration inequality for log Z(beta,N)}), we invoke the Borel--Cantelli lemma, and then it tells $\rho(\beta)=\Tilde{\rho}(\beta)$ for all $\beta\geq0$. By Jensen's inequality, we then can conclude that $\rho(\beta)\leq\lambda(\beta)$ for all $\beta\geq0$.\par\noindent
        \textbf{Step IV.} Monotonicity of $\beta\mapsto\rho(\beta)-\lambda(\beta)$.\par\noindent
        The continuity of $\rho(\vdot)-\lambda(\vdot)$ on $[0,\infty)$ follows immediately from the convexity of $\lim_{N\to\infty} N^{-1}\mathbb{E}[\log Z^{\beta,\omega}_N]$ in $\beta$. So we only need to show $\beta\mapsto\rho(\beta)-\lambda(\beta)$ is non-increasing in $\beta$. Combining (\ref{eqn: L2 one}
        ) with (\ref{eqn: L2 two}), and together with
        \[
            \big|\frac{\partial}{\partial\beta}\log W^{\beta,\omega}_N \big| \leq \big| (W^{\beta,\omega}_N)^{-1} \big| \vdot \big| \frac{\partial}{\partial\beta}W^{\beta,\omega}_N \big|,\qquad\forall~\beta\geq0,
        \]
        we know that $\normx{\frac{\partial}{\partial\beta} \log W^{\beta,\omega}_N }_T$ is in $L^1(\mathbb{P})$ for each $T>0$. Here the notation $\norm{\vdot}_T$ is as in Lemma \ref{lem: L2 space for processes L and W}. Since $\log W^{\beta,\omega}_N$ is in $C^1(\mathbb{R}_+)$ in $\beta\in\mathbb{R}_+$, by Fubini's theorem we have
        \[
            \mathbb{E}\big[\log W^{\beta,\omega}_N\big] = \int_0^\beta \mathbb{E}\big[ \frac{\partial}{\partial \theta} \log W^{\theta,\omega}_N \big]\,d\theta,\qquad\forall~\beta\geq0.
        \]
        Therefore, 
        \begin{equation}\label{eqn: kkkk}        
            \frac{\partial}{\partial\beta}  \mathbb{E} \big[\log W^{\beta,\omega}_N\big] =   \mathbb{E}\big[ \frac{\partial}{\partial\beta} \log W^{\beta,\omega}_N \big],\qquad\forall~\beta\geq0\quad\text{and}\quad N\in\mathbb{N}. 
        \end{equation}
        If we define the auxiliary field $\omega^\dagger$ by $\omega^\dagger_{n,z}\coloneqq\int_{\mathbb{N}\backslash\{n\}\times\mathbb{Z}^d}\omega_{n,z}\,d\mathbb{P}$ for the time-correlated field $\omega$, the time-correlation is integrated in $\omega^\dagger$, then by Lemma \ref{lem: kappa-bound for a cone} and repeating the steps above (\ref{eqn: L2 two}), we have
        \begin{equation}\label{eqn: find a name for eqn}                    
            \mathbb{E}\big[\frac{\partial}{\partial\beta}\log W^{\beta,\omega}_N\big] \leq \kappa^{2N} E^S_0 \mathbb{E}\bigg[ \big(W^{\beta,\omega^\dagger}_N \big)^{-1}\big(\sum_{k=1}^N \omega^\dagger_{k,S_k} -\frac{1}{\mathbb{E}[Z^{\beta,\omega^\dagger}_N]}\frac{\partial}{\partial\beta} \mathbb{E}[Z^{\beta,\omega^\dagger}_N]  \big) \frac{e^{\beta \sum_{k=1}^N \omega^\dagger_{k,S_k} } } {\mathbb{E}[Z^{\beta,\omega^\dagger}_N]} \bigg],
        \end{equation}
        where $\mathbb{E}[Z^{\beta,\omega^\dagger}_N]^{-1}e^{\beta \sum_{k=1}^N \omega^\dagger_{k,S_k}}\,\mathbb{P}(d\omega)$ is product across the strips $\{n\}\times\mathbb{Z}^d$, $n=1,\ldots,N$. And whence this product probability measure satisfies the FKG inequality \cite[p. 78]{Liggett}. Because the function $\sum_{k=1}^N\omega^\dagger_{k,S_k} -\mathbb{E}[Z^{\beta,\omega^\dagger}_N]^{-1}\frac{\partial}{\partial\beta} \mathbb{E}[Z^{\beta,\omega^\dagger}_N] $ is increasing in $\omega$ and because $\mathbb{E}[Z^{\beta,\omega^\dagger}_N]^{-1}$ is decreasing, it is yielded that
        \begin{equation*}\begin{aligned}
            &E_{\mathbb{P}\otimes Q}\bigg[ \big(W^{\beta,\omega^\dagger}_N\big)^{-1} \big(\sum_{k=1}^N \omega^\dagger_{k,S_k} -\frac{1}{\mathbb{E}[Z^{\beta,\omega^\dagger}_N]}\frac{\partial}{\partial\beta} \mathbb{E}[Z^{\beta,\omega^\dagger}_N]  \big) \frac{e^{\beta \sum_{k=1}^N \omega^\dagger_{k,S_k} } }{\mathbb{E}[Z^{\beta,\omega^\dagger}_N]} \bigg]\\ 
            &\qquad\leq \mathbb{E}\bigg[ \big(W^{\beta,\omega^\dagger}_N\big)^{-1}  \frac{e^{\beta \sum_{k=1}^N \omega^\dagger_{k,S_k} } }{\mathbb{E}[Z^{\beta,\omega^\dagger}_N]} \bigg] \vdot \mathbb{E}\bigg[ \big(\sum_{k=1}^N\omega^\dagger_{k,S_k} -\frac{1}{\mathbb{E}[Z^{\beta,\omega^\dagger}_N]}\frac{\partial}{\partial\beta} \mathbb{E}[Z^{\beta,\omega^\dagger}_N]  \big) \frac{e^{\beta \sum_{k=1}^N \omega^\dagger_{k,S_k} } }{\mathbb{E}[Z^{\beta,\omega^\dagger}_N]} \bigg],
        \end{aligned}\end{equation*}
        where the right-hand side of the above expression vanishes. In light of (\ref{eqn: find a name for eqn}),
        \[
            \frac{1}{N} \mathbb{E}\big[\frac{\partial}{\partial\beta}\log W^{\beta,\omega}_N\big] \leq \kappa^{2N}\vdot0\leq0,\qquad\forall~\beta\geq0\quad\text{and}\quad N\in\mathbb{N}.
        \]
        Hence by (\ref{eqn: kkkk}), for any $0\leq\beta_1<\beta_1<\infty$,
        \[
            \mathbb{E}\big[\log W^{\beta_1,\omega}_N\big] \geq \mathbb{E}\big[\log W^{\beta_2,\omega}_N\big]\qquad\Longrightarrow\quad \lim_{N\to\infty}\frac{1}{N} \mathbb{E}\big[\log W^{\beta_2,\omega}_N\big] - \lim_{N\to\infty}\frac{1}{N} \mathbb{E}\big[\log W^{\beta_1,\omega}_N\big] \leq 0,
        \]
        verifying the assertion.
    \end{proof}
    To show the existence of the annealed free energy $\lambda(\beta)$ for any $\beta\geq0$, one could alternatively use (the more cumbersome) Bryc's formula \cite[Theorem 4.4.2]{Dembo/Zeitouni} for the existence of annealed large deviation principle with good rate function. And then use Varadhan's lemma with exponential tightness to subsequently verify the free energy exists as its Fenchel--Legendre transform. Though we do not aim to follow this path, it nevertheless sheds new light on the overall statistical mechanics structure of the directed polymer models.

%-----------------------------------------
%-----------------------------------------
\section{Localization regime}\label{sec: localization}
    Define $W^{\beta,\omega}_N\coloneqq\frac{Z^{\beta,\omega}_N}{\mathbb{E}[Z^{\beta,\omega}_N]}$, $W^{\beta,\omega(l)}_N\coloneqq\frac{Z^{\beta,\omega(l)}_N}{\mathbb{E}[Z^{\beta,\omega(l)}_N]}$ and $W^{1,\xi(l)}_N\coloneqq\frac{Z^{1,\xi(l)}_N}{\mathbb{E}E_Q[Z^{1,\xi(l)}_N]}$ for each $l\geq1$ and for all $\beta\geq0$. It is clear that $E_Q[W^{1,\xi(l)}_N] = W^{\beta,\omega(l)}_N$, $\mathbb{P}$-a.s. And we define $\mathcal{L}_{n,\beta}(\omega,\epsilon)\coloneqq W^{1,\xi(l)}_{\tau^{(L)}_n}$ for each $n\in\mathbb{N}$ with the index $l\geq1$ implicitly assumed. Now, using the above separation lemma for the time-correlation conditions, we are above to prove the existence of localization regime, i.e.~$\beta_*>0$.\par
    We will also need the \textit{approximate martingale} $\mathcal{H}_{n,\beta}(\omega,\epsilon)\coloneqq\mathcal{L}_{n,\beta}(\omega^*,\epsilon)$ for each $n\in\mathbb{N}$ and $\beta\geq0$. Here, by writing $\omega^*$ we refer to the underlying field $(\omega^*_{n,z})_{(n,z)\in\mathbb{N}\times\mathbb{Z}^d}$ obtained by $\omega^*_{n,z}\coloneqq\int_{H(k)}\omega_{n,z}\,d\mathbb{P}$ with $H(k)\coloneqq(\mathbb{N}\backslash(\tau^{(L)}_{k-1},\tau^{(L)}_k])\times\mathbb{Z}^d$ for some $k$. Namely, $\omega^*$ is created via marginally integrating the $\omega$-coordinates outside the strip $(\tau^{(L)}_{m-1},\tau^{(L)}_m]\times\mathbb{Z}^d$ when $\tau^{(L)}_{m-1}<m\leq\tau^{(L)}_m$, whence cancels the distant time-correlation effect. See also Appendix Lemma \ref{lem: tau mixing inequality} for a discussion on this constructive point of view.\par
    It is easily observed that $\mathbb{E}E_Q [\mathcal{H}_{n,\beta}|\mathscr{G}_{n-1}]=\mathcal{H}_{n-1,\beta}(\omega,\epsilon)$ for each $n\geq1$, justifying its name. And whence for this nonnegative $(\mathscr{G}_n)_{n\geq1}$-martingale $(\mathcal{H}_{n,\beta})_{n\geq1}$, we have the convergence
    \[
        \mathcal{H}_{n,\beta}(\omega,\epsilon)\xrightarrow[]{n\to\infty}\mathcal{H}_{\infty,\beta}(\omega,\epsilon),\qquad\mathbb{P}\otimes Q\text{-a.s.}
    \]
    to a nonnegative limiting $\mathcal{H}_{\infty,\beta}$. It also follows that
    \[
        \mathbb{E}E_Q[\mathcal{H}_{n,\beta}] = \mathbb{E} E_Q [\mathcal{H}_{1,\beta}] = 1,\qquad\forall~n\geq1.
    \]
    It is the next lemma where we need the transient dimension, i.e.~$d\geq3$, for the reference walk $(S_n)_{n\geq1}$. Remark that we define the continuous function $\Lambda:\beta\in[0,\infty)\mapsto \log \kappa_1(\beta)\kappa_2(\beta) + \log\mathbb{E}[e^{2\beta\omega_{1,0}}] + 2\log\mathbb{E}[e^{-\beta\omega_{1,0}}]$ in Appendix \ref{lem: function Lambda(beta)}.

    %lemma: square-integrable martingale
    \begin{lemma}\label{lem: square-integrable martingale}
        \normalfont
        Given $d\geq3$ and under the time-correlations $(\textbf{TC})_{C,g}$ or $(\textbf{TCG})_{C,g}$, we have
        \[
            \Lambda(\beta) < K \qquad\Longrightarrow\quad \sup_{n\geq1}\mathbb{E}E_Q\big[(\mathcal{H}_{n,\beta})^2\big] < \infty,
        \]
        for some absolute constant $K>0$ specified in Appendix \ref{appendix b}.
    \end{lemma}
    \begin{proof}
        The proof of this lemma is not simple, and is deferred to Appendix \ref{appendix b}.
    \end{proof}

    %lemma: positive H limit
    \begin{lemma}\label{lem: positive H limit}
        \normalfont
        Given $d\geq3$ and under time-correlations, for any $\beta\geq0$ such that $\Lambda(\beta)<K$, we have $\mathcal{H}_{\infty,\beta}(\omega,\epsilon)>0$ and $E_Q[\mathcal{H}_{\infty,\beta}(\omega,\vdot)]>0$, $\mathbb{P}\otimes Q$-a.s.
    \end{lemma}
    \begin{proof}
        A straightforward computation shows that
        \[
            \mathcal{H}_{n,\beta}(\omega,\epsilon) = \sum_{k\geq1}\sum_{S^\prime\in\mathcal{S}_k}\frac{Q(\tau^{(L)}_1=k)P^S_0(S_j=S^\prime_j,\,j=1,\ldots,k)}{\mathbb{E}E_Q[Z^{1,\xi(l)}_{\tau^{(L)}_1}]} e^{\sum_{j=1}^k \xi^{l,*}_{j,S^\prime_j} } \mathcal{H}_{n-1,\beta}(\vartheta_{k,S^\prime_k}\omega,\iota_k\epsilon),
        \]
        where $\vartheta$ and $\iota$ respectively denote the shift operators on $\omega$ and $\epsilon$. Henceforth, the event $\{\mathcal{H}_{\infty,\beta}=0\}$ is translation invariant and thus, see \cite{Bolthausen},
        \begin{equation}\label{eqn: dog??}
            \mathbb{P}\otimes Q(\mathcal{H}_{\infty,\beta}=0)\in\{0,1\},\qquad\forall~\beta\geq0.
        \end{equation}
        Applying Lemma \ref{lem: square-integrable martingale} to $\beta\geq0$ satisfying $\Lambda(\beta)< K$, we know $(\mathcal{H}_{n,\beta})_{n\geq1}$ is a square-integrable $(\mathscr{G}_n)_{n\geq1}$-martingale. And thus in light of (\ref{eqn: dog??}) and that
        \[
            \mathbb{E}E_Q[\mathcal{H}_{\infty,\beta}]=\lim_{n\to\infty} \mathbb{E}E_Q[\mathcal{H}_{n,\beta}] =1,
        \]
        the assertion is yielded.
    \end{proof} 
    Invoking Lemma \ref{lem: mixing inequality for martingale H} as well as Lemma \ref{lem: tau mixing inequality} again, we have
    \[
        \exp(-Cne^{-gtL}) E_Q [\mathcal{H}_{n,\beta}(\omega,\vdot)] \leq E_Q [\mathcal{L}_{n,\beta}(\omega,\vdot)] \leq \exp(C^\prime n e^{-gtL}) E_Q [\mathcal{H}_{n,\beta}(\omega,\vdot)],\qquad\mathbb{P}\text{-a.s.}
    \]
    And whence taking limit $n\to\infty$, it yields
    \begin{equation}\label{eqn: dog later}    
        -Ce^{-gtL}\leq\varliminf_{n\to\infty}\frac{1}{n}  \log E_Q [\mathcal{L}_{n,\beta}(\omega,\vdot)] \leq \varlimsup_{n\to\infty}\frac{1}{n}  \log E_Q [\mathcal{L}_{n,\beta}(\omega,\vdot)] \leq C^\prime e^{-gtL},\qquad\mathbb{P}\otimes Q\text{-a.s.}
    \end{equation}
    Remark that (\ref{eqn: dog later}) holds without taking $Q$-expectation either, by simple observation from Lemma \ref{lem: tau mixing inequality} again.
    \begin{proof}[Proof of Theorem \ref{thm: main results for beta*>0}.]
        By the law of large numbers,
        \[
            \frac{\tau^{(L)}_n}{n}\xrightarrow[]{n\to\infty} E_Q\big[\tau^{(L)}_1\big],\qquad Q\text{-a.s.}
        \]
        Take a nondecreasing sequence $(k_N)_{N\geq1}$ such that $\tau^{(L)}_{k_N}\leq N<\tau^{(L)}_{k_N+1}$ for any $N\geq1$. It is also clear that 
        \[
            \frac{k_N}{N}\xrightarrow[]{N\to\infty}\frac{1}{E_Q[\tau^{(L)}_1]},\qquad Q\text{-a.s.}
        \]
        Furthermore, by the choice of $l\geq1$ we have
        \[
            \log Z^{1,\xi(l)}_N\geq e^{-\beta l(N-\tau^{(L)}_{k_N})} Z^{1,\xi(l)}_{\tau^{(L)}_{k_N} },\qquad\mathbb{P}\otimes Q\text{-a.s.}
        \]
        On the other hand,
        \[
            \mathbb{E}\big[Z^{\beta,\omega}_N\big]\leq e^{(N-\tau^{(L)}_{k_N})(\ln\kappa+\ln \mathbb{E}[e^{\beta\omega_{1,0}}])} \mathbb{E}\big[Z^{\beta,\omega}_{\tau^{(L)}_{k_N} }\big],\qquad\forall~N\geq1.
        \]
        Notice also that
        \[
            \frac{1}{N}\big(\tau^{(L)}_{k_N+1} - \tau^{(L)}_{k_N}\big) = \frac{\tau^{(L)}_{k_N+1}}{k_N+1}\frac{k_N+1}{N} - \frac{\tau^{(L)}_{k_N}}{k_N}\frac{k_N}{N} \xrightarrow[]{N\to\infty}0,\qquad\mathbb{P}\otimes Q\text{-a.s.}
        \]
        Therefore, with (\ref{eqn: dog later}) and the remark below, letting $N\to\infty$, we get
        \[
            \varliminf_{N\to\infty} \frac{1}{N}\log W^{1,\xi(l)}_N \geq0,\qquad\mathbb{P}\otimes Q\text{-a.s.}
        \]
        Hence, for any $\delta>0$, there exists $N_0^\prime(\omega,\epsilon)\in\mathbb{N}$ such that $N^{-1}\log W^{1,\xi(l)}_N\geq-\delta$ whenever $N\geq N_0^\prime$. Thus, we can take sufficiently large $N_0(\omega)\in\mathbb{N}$ such that $Q(N_0(\omega)\geq N_0^\prime)\geq1-\delta$. Then,
        \[
            \frac{1}{N} E_Q\big[\log W^{1,\xi(l)}_N \big]\geq-(1+l)\delta,\qquad\forall~N\geq N_0(\omega),\qquad\mathbb{P}\text{-a.s.}
        \]
        Letting first $N\to\infty$ and then $\delta\to0$, we get
        \begin{equation*}
            \varliminf_{N\to\infty} \frac{1}{N} \log    E_Q\big[W^{1,\xi(l)}_N \big] \geq \varliminf_{N\to\infty} \frac{1}{N} E_Q\big[\log W^{1,\xi(l)}_N \big] \geq0,\qquad\mathbb{P}\text{-a.s.}
        \end{equation*}
        On the other hand, Jensen's inequality gives
        \begin{equation}\label{eqn: dog what}
            \lim_{N\to\infty} \frac{1}{N} \log    E_Q\big[W^{1,\xi(l)}_N \big] = 0,\qquad\mathbb{P}\text{-a.s.}
        \end{equation}
        Here (\ref{eqn: dog what}) is equivalent to say $\rho_l(\beta)=\lambda_l(\beta)$ for all $l\geq1$, where $\rho_l(\beta)\coloneqq\lim_{N\to\infty}N^{-1}\log Z^{\beta,\omega(l)}_N$ and $\lambda_l(\beta) \coloneqq \lim_{N\to\infty}N^{-1}\log \mathbb{E}[ Z^{\beta,\omega(l)}_N]$ for any $l\geq1$ and $\beta\geq0$. By Lemma \ref{lem: rho_l to rho, lambda_l to lambda}, we know $\rho_l(\beta)\to\rho(\beta)$ and $\lambda_l(\beta)\to\lambda(\beta)$ as $l\to\infty$. Hence the theorem is verified.        
    \end{proof}

%-----------------------------------------
%-----------------------------------------
\section{Delocalization regime}
    In this section, we show that under suitable conditions, the delocalization regime coexists with the localization regime non-trivially. We follow the convention in Lemma \ref{lem: annealed free energy is differentiable}, namely, 
    \[
            I_N\coloneqq e^{\beta\sum_{k=1}^N\omega_{k,S_k}},\quad I_{i,j}\coloneqq e^{\beta(\omega_{i,S_i}+\omega_{j,S_j})},\quad\text{and}\quad I_N^{i,j}\coloneqq e^{\beta\sum_{k\neq i,j}\omega_{k,S_k}},\qquad\forall~1\leq i,j\leq N.
    \]
    And we denote the law $\hat{\mathbb{E}}$ by specifying its density $d\hat{\mathbb{P}}/d\mathbb{P}=\frac{1}{\mathbb{E}[e^{\beta\omega_{1,0}}]}e^{\beta\omega_{1,0}}$ with respect to $\mathbb{P}$. The following lemma extends \cite[Lemma A.1]{Viveros} to time-correlated field and does not require a long-range structure for the reference random walk. 
    %lemma: limit of entropy inequality
    \begin{lemma}\label{lem: limit of entropy inequality}
        \normalfont
        Under the time-correlation, 
        \[
            \log\frac{1}{\mathbb{P}(\omega_{1,0}=\hbar)}-\log\kappa^2 \leq \varliminf_{\beta\to\infty} \beta\frac{d\lambda}{d\beta}(\beta)-\lambda(\beta) \leq \varlimsup_{\beta\to\infty} \beta\frac{d\lambda}{d\beta}(\beta)-\lambda(\beta) \leq \log\frac{1}{\mathbb{P}(\omega_{1,0}=\hbar)} +\log\kappa^2.
        \]
    \end{lemma}
    \begin{proof}
        Let us denote the law $\mathbb{E}^\beta_N$ by specifying its density $d\mathbb{P}^\beta_N/d\mathbb{E}P^S_0=\frac{1}{\mathbb{E}[Z^{\beta,\omega}_N]}I_N$. For each $\beta\geq0$ and for any $h<\hbar$ choose $h<h^\prime<\hbar$, then there exists $\delta>0$ such that
        \[
            \frac{1}{N}\sum_{k=1}^N\mathbb{E}P^S_0(\omega_{k,S_k}>h^\prime) = \mathbb{P}(\omega_{1,0}>h^\prime)>\delta>0.
        \]
        And then
        \[
            \frac{1}{N}\sum_{k=1}^N\mathbb{P}^\beta_N(\omega_{k,S_k}<h)\leq \frac{\kappa^2}{N}\sum_{k=1}^N \hat{\mathbb{P}}(\omega_{1,0}<h)\leq\kappa^2\frac{e^{\beta h}}{\mathbb{E}[e^{\beta\omega_{1,0}},\,\omega_{1,0}>h^\prime]}\leq \kappa^2\delta^{-1}e^{-\beta(h^\prime-h)}.
        \]
        This actually implies that, for any $h<\hbar$,
        \begin{equation}\label{eqn: uniform convergence in N}
            \frac{1}{N}\sum_{k=1}^N\mathbb{P}^\beta_N(\omega_{k,S_k}<h) \rightrightarrows0,\quad\text{as}\;\;\beta\to\infty,\quad\text{uniformly in}~N.
        \end{equation}
        Without loss of generality, let us now assume $\hbar<\infty$. For any $\delta>0$, choose $h>0$ with $0<\hbar-h<\delta$. Notice that by (\ref{eqn: pointwise convergence of derivative of lambda}) we have
        \[
            \lambda^\prime(\beta)=\lim_{N\to\infty}\lambda^\prime_N(\beta)=\lim_{N\to\infty}\frac{1}{N}\sum_{k=1}^N \mathbb{E}^\beta_N[\omega_{k,S_k}]\leq\hbar.
        \]
        And for any $N\geq1$, we have 
        \[
            \frac{1}{N}\sum_{k=1}^N\mathbb{E}^\beta_N[\omega_{k,S_k}] \geq \frac{1}{N}\sum_{k=1}^N\mathbb{E}^\beta_N[\omega_{k,S_k}\mathbbm{1}_{\{\omega_{k,S_k}\geq h\}}] \geq \frac{1}{N}(\hbar-\delta)\sum_{k=1}^N\mathbb{P}^\beta_N(\omega_{k,S_k}\geq h)\geq (\hbar-\delta)(1-\delta),
        \]
        where for the last inequality, we have chosen $\beta$ sufficiently large uniformly in $N$, see (\ref{eqn: uniform convergence in N}). Therefore,
        \begin{equation}\label{eqn: limit i beta of derivative}
            \lim_{\beta\to\infty} \lambda^\prime(\beta)=\hbar.
        \end{equation}
        For any $\delta>0$,  choose $h>0$ such that $0<\hbar-h<\delta$ and $\mathbb{P}(\omega_{1,0}\geq h)\leq\delta+\mathbb{P}(\omega_{1,0}=\hbar)$. Then,
        \[
            e^{\beta\hbar N-\lambda_N(\beta)N}\big(\mathbb{P}(\omega_{1,0}=\hbar)+\delta\big)^N\geq \kappa^{-N}\mathbb{E}E^S_0\big[I_N e^{-\lambda_N(\beta)N}\mathbbm{1}_{ \cap_{k=1}^N\{ \omega_{k,S_k}\geq h \}}\big]\geq\kappa^{-2N}\hat{\mathbb{P}}(\omega_{1,0}\geq h)^N,
        \]
        where using a derivation similar to (\ref{eqn: uniform convergence in N}), we know $\hat{\mathbb{P}}(\omega_{1,0}\geq h)\to1$ as $\beta\to\infty$. Henceforth, 
        \[
            \beta\hbar-\lambda_N(\beta)\geq \log \hat{\mathbb{P}}(\omega_{1,0}\geq h) -\log\kappa^2 + \log\frac{1}{\mathbb{P}(\omega_{1,0}=\hbar)+\delta}.
        \]
        First letting $N\to\infty$, then $\beta\to\infty$, and then letting $\delta\to0$, we obtain the lower-bound. The upper-bound can be derived in a similar fashion, and we omit the details here.     
    \end{proof} 

    \begin{proof}[Proof of Theorem \ref{thm: main results for beta*<infinity}.]
        For any $\beta\geq0$,
        \[
            Z^{\beta,\omega}_N=\sum_{x\in\mathbb{Z}^d}P^S_0(S_1=x) e^{\beta\omega_{1,x}} Z^{\beta,\vartheta_{1,x}\omega}_{N-1}.
        \]
        Let $1<\theta<1$, by the subaddictive estimate,
        \[
            (Z^{\beta,\omega}_N)^\theta\leq\sum_{x\in\mathbb{Z}^d} P^S_0(S_1=x)^\theta e^{\beta\theta\omega_{1,x}}(Z^{\beta,\vartheta_{1,x}\omega}_{N-1})^{\theta}.
        \]
        Inductively, we have
        \begin{equation*}\begin{aligned}
            (Z^{\beta,\omega}_N)^\theta &\leq \sum_{S^\prime\in\mathcal{S}_N} e^{\beta\theta\sum_{k=1}^N\omega_{k,S^\prime_k}} \prod_{k=1}^N P^S_0(S_k-S_{k-1}=S^{\prime}_k-S^\prime_{k-1})^\theta\\
            &\leq p_S^{(\theta-1)N} \sum_{S^\prime\in\mathcal{S}_N} e^{\beta\theta\sum_{k=1}^N\omega_{k,S^\prime_k}} P^S_0(S_k=S^\prime_k,\,k=1,\ldots,N),
        \end{aligned}\end{equation*}
        where $P_S\coloneqq\min\{P^S_0(S_1=x):\,x\in\mathbb{Z}^d\,\text{s.t.}\,P^S_0(S_1=x)>0\}$. Notice that by the finite-range condition $\normx{S_1}_1<\infty$, we always have $p_S>0$. Define the constant $K(S)\coloneqq-\log p_S/H(S_1)$, where $H(S_1)$ is the discrete entropy of the distribution of $S_1$, see \cite[Definition 2.1]{Cover}. Hence,
        \[
            \mathbb{E}\big[(Z^{\beta,\omega}_N)^\theta\big]\leq p_S^{(\theta-1)N}\mathbb{E}\big[Z^{\beta\theta,\omega}_N\big]
        \]
        and
        \[
            \rho(\beta)=\lim_{N\to\infty}\frac{1}{\theta N}\mathbb{E}\big[\theta\log Z^{\beta,\omega}_N\big] \leq \varlimsup_{N\to\infty} \frac{1}{\theta N}\log \mathbb{E}\big[(Z^{\beta,\omega}_N)^\theta\big],
        \]
        which implies
        \begin{equation}\label{eqn: dog 5.3}        
            \rho(\beta)\leq\inf_{0<\theta<1}\big\{ \theta^{-1}\mathscr{H}(\theta) \big\},\qquad\text{with}\quad\mathscr{H}(\theta)\coloneqq \log p_S^{\theta-1}+\lambda(\beta\theta).
        \end{equation}
        If $\frac{d}{d\theta}\theta^{-1}\mathscr{H}(\theta)>0$ at $\theta=1$, the infimum (\ref{eqn: dog 5.3}) above is achieved at some $0<\theta<1$, and will be strictly less than $\lambda(\beta)$. Indeed, 
        \[
            \frac{d}{d\theta} \theta^{-1}\mathscr{H}(\theta) = \theta^{-2}\big( \log p_S + \beta\theta\lambda^\prime(\beta\theta) - \lambda(\beta\theta) \big)>0
        \]
        at $\theta=1$ whenever $\beta\lambda^\prime(\beta)-\lambda(\beta)>K(S)H(S_1)$. Furthermore, by Lemma \ref{lem: limit of entropy inequality} we know $\varliminf_{\beta\to\infty} \beta\lambda^\prime(\beta)-\lambda(\beta)\geq K^\prime - \log\mathbb{P}(\omega_{1,0}=\hbar) $, where we take $K^\prime=-2\log\kappa$. And then the assertion is verified.
    \end{proof}

%-----------------------------------------
%-----------------------------------------
\appendix

%-----------------------------------------
\section{$L^2$-integrable martingale}\label{appendix b}
    In this section we prove the process $(\mathcal{H}_{n,\beta})_{n\geq1}$ is a square-integrable martingale with sufficiently small $\beta>0$ for each $l\geq1$. And we will abbreviate $\tau^{(L)}_n$ by $\tau_n$ for each $n\in\mathbb{N}$. Define $\mathcal{H}_{n,\beta}$ as in (\ref{eqn: definition of H}) and we denote $\psi(\beta)\coloneqq\mathbb{E}E_Q[Z^{1,\xi(l)}_{\tau_1}]=\mathbb{E}\overline{E}^S_0[e^{\sum_{k=1}^{\tau_1}\xi^l_{k,S_k}}  ]$ for all $\beta>0$. Let $(\widetilde{S}_n)_{n\geq}$ be an i.i.d. copy of $(S_n)_{n\geq1}$, we can express $\normx{\mathcal{H}_{n,\beta}}_{L^2(\mathbb{P}\otimes Q)}$ by
    \[
        \mathbb{E}E_Q\big[(\mathcal{H}_{n,\beta})^2\big] = \psi(\beta)^{-2n} E^{S\otimes\widetilde{S}}_{0,0}\prod_{j=0}^{n-1}\mathbb{E}E_Q\big[ e^{\Delta_{j,n}+\widetilde{\Delta}_{j,n}} e^{\sum_{k=\tau_j+1}^{\tau_{j+1}} \xi^l_{k,S_k} + \sum_{k=\tau_j+1}^{\tau_{j+1}} \xi^l_{k,\widetilde{S}_k} } \big],
    \]
    where $\widetilde{\Delta}_{j,n}\coloneqq\log\mathbb{E}\overline{E}^{\widetilde{S}}_0[e^{\sum_{k=\tau_j+1}^{\tau_{j+1}}\xi^l_{k,\widetilde{S}_k} } |\mathscr{G}_j] - \log\psi(\beta)$ for all $0\leq j\leq n-1$. We also set $\Vec{X}_j \coloneqq (S_{\tau_{j-1}+1},\ldots,S_{\tau_j})$ and $\Vec{Y}_j \coloneqq (\widetilde{S}_{\tau_{j-1}+1},\ldots,\widetilde{S}_{\tau_j})$ for each $1\leq j\leq n$, and 
    \[
        \frac{1}{2}V(\Vec{X}_j,\Vec{Y}_j)\coloneqq \log \frac{\mathbb{E}[e^{\Delta_{j,n}+\sum_{k=\tau_j+1}^{\tau_{j+1}} \xi^l_{k,S_k}} e^{\widetilde{\Delta}_{j,n}+\sum_{k=\tau_j+1}^{\tau_{j+1}} \xi^l_{k,\widetilde{S}_k}}]}{\mathbb{E}E_Q[e^{\Delta_{j,n}+\sum_{k=\tau_j+1}^{\tau_{j+1}}\xi^l_{k,S_k}}] \mathbb{E}E_Q[e^{\widetilde{\Delta}_{j,n}+\sum_{k=\tau_j+1}^{\tau_{j+1}} \xi^l_{k,\widetilde{S}_k}}]}.
    \]
    And we have the following assertion, which is adapted and refined from \cite[Lemma 3.3]{Bazaes/Mukherjee/Ramirez/Saglietti}.
    %lemma: bound for ball probability
    \begin{lemma}\label{lem: bound for ball probability}
        \normalfont
        For any $r>\sqrt{d}$, 
        \[
            \sum_{j=0}^\infty P^{S\otimes\widetilde{S}}_{x,y}\otimes Q\big(d_1(\Vec{X}_j,\Vec{Y}_j)\leq r\big) \leq \sum_{j=0}^\infty \overline{E}^{S\otimes\widetilde{S}}_{x,y}\big[N^r_j(\Vec{X},\Vec{Y})\mathbbm{1}_{\{d_1(\Vec{X}_j,\Vec{Y}_j)\leq r\}}\big]\leq K_1r^2
        \]
        for some constant $K_1>0$ and for any $x,y\in\mathbb{Z}^d$, where
        \[
            N^r_j(\Vec{X},\Vec{Y})\coloneqq\sum_{k=\tau_{j-1}+1}^{\tau_j} \mathbbm{1}_{\{ |S_k-\widetilde{S}_k|_1\leq r \}},\qquad\forall~1\leq j\leq n.
        \]
    \end{lemma}
    \begin{proof}
        Denote $x-y\eqqcolon z$ and $S_n-\widetilde{S}_n\eqqcolon Z_n$ for each $n\in\mathbb{N}$. Then $(Z_n)_{n\geq0}$ is a finite-range simple random walk with law denoted by $P_z$ and $E_z\coloneqq E_{P_z}$. It is clear that
        \[
            \sum_{j=0}^\infty \overline{E}^{S\otimes\widetilde{S}}_{x,y}\big[N^r_j(\Vec{X},\Vec{Y})\mathbbm{1}_{\{d_1(\Vec{X}_j,\Vec{Y}_j)\leq r\}}\big] = \sum_{n=0}^\infty E_z\big[\mathbbm{1}_{\{\abs{Z_n}_1\leq r\}}\big].
        \]
        Since we are interested in the event $\{\abs{Z_n}_1\leq r\}$, we need to consider the function $\prod_{j=1}^d f(x_j/r)$ where we define $f(x_j)\coloneqq\max\{1-\abs{x_j},0\}$. Then we have the Fourier transform of the product
        \[
            \widehat{\prod_{j=1}^d f(x_j) } = \prod_{j=1}^d \widehat{f}(\xi_j),\qquad\text{with}\quad \widehat{f}(\xi_j)=\frac{2}{\xi_j^2}(1-\cos\xi_j).
        \]
        Let $\mu$ denote the law of $Z_1=S_1-\widetilde{S}_1$ and $\mu^{\star n}=\mu\star\cdots\star\mu$ its $n$-fold convolution. Then,
        \[
            \int_{\mathbb{R}^d}\widehat{\prod_{j=1}^d f(x_j/r)}\,\mu^{\star n}(dx) = r^d\int_{\mathbb{R}^d} \prod_{j=1}^d f(\xi_j r)\chi_\mu(\xi)^n\,d\xi,
        \]
        where the characteristic function $\chi_\mu(\xi)\coloneqq E_z[e^{i\langle\xi,Z_1-Z_0\rangle}]$ takes only real non-negative value because $Z_1-z=(S_1-x)-(\widetilde{S}_1-y)$ with $S_1-x,\widetilde{S}_1-y$ i.i.d. Hence for any $0<\theta<1$,
        \[
            \int_{\mathbb{R}^d} \widehat{\prod_{j=1}^d f(x_j/r)} \sum_{n=0}^\infty \theta^n\mu^{\star n}(dx) = r^d \int_{\mathbb{R}^d} \frac{\prod_{j=1}^d f(\xi_j r)}{1-\theta\chi_\mu(\xi)}\,d\xi,
        \]
        which, with $B_r(0)\coloneqq\{x\in\mathbb{R}^d:\,\abs{x}_1\leq r\}$, implies that
        \begin{equation*}\begin{aligned}
            &\sum_{n=0}^\infty P_z\big(\abs{Z_n}_1\leq r\big) \leq \sum_{n=0}^\infty \mu^{\star n}\big(B_r(0)\big) \leq C\int_{\mathbb{R}^d} \widehat{\prod_{j=1}^d f(x_j/r) } \sum_{n=0}^\infty \mu^{\star n}(dx)\\
            &\qquad\leq Cr^d\sup_{0<\theta<1} \int_{\mathbb{R}^d} \frac{\prod_{j=1}^d f(\xi_j r)}{1-\theta\chi_\mu(\xi)}\,d\xi \leq C^\prime r^d \int_{B_{\delta}(0)} \frac{d\xi}{1-\chi_\mu(\xi)},\qquad\text{where}\quad\delta\coloneqq\tfrac{\sqrt{d}}{r}.
        \end{aligned}\end{equation*}
        By Taylor's expansion, we have $\chi_\mu(\xi)\leq 1-\frac{1}{2}\sum_{j,k=1}^d a_{jk}\xi_j\xi_k+C\abs{\xi}^3$ for some constant $C>0$, where $(a_{jk})_{j,k}$ is the covariance matrix of $Z_1$. Finally, there exists constant $c_0>0$ such that $c_0\abs{\xi}^2\leq 1-\chi_\mu(\xi)$ if $\abs{\xi}\leq1$. In particular, when $r>\sqrt{d}$,
        \[
            \sum_{n=0}^\infty P_z\big(\abs{Z_n}_1\leq r\big) \leq C r^d \int_{B_\delta(0)} \abs{\xi}^{-2}\,d\xi\leq C^\prime r^d \int_r^\infty t^{-(d-1)}\,dt\leq K_1 r^2,
        \]
        verifying the assertion.
    \end{proof}

    %lemma: probability estimate
    \begin{lemma}\label{lem: probability estimate}
        \normalfont
        Let $(Z_n)_{n\geq0}$ be a Markovian random walk on $\mathbb{Z}^d$ starting from any $z\in\mathbb{Z}^d$ with law $P_z$ and $E_z\coloneqq E_{P_z}$. If we define
        \[
            \eta(r)\coloneqq\sum_{n=0}^\infty E_0\big[\mathbbm{1}_{\{ \abs{Z_n}_1\leq r \}}\big] = \sum_{n=0}^\infty P_0\big(\abs{Z_n}_1\leq r\big),\qquad\forall~r>0,
        \]
        then for any $C>0$ with $C\eta(r)<1$, we have
        \[
            E_0\big[e^{C \sum_{n=0}^\infty \mathbbm{1}\{\abs{Z_n}_1\leq r\} }\big] \leq \frac{1}{1-C\eta(r)}.
        \]
    \end{lemma}
    \begin{proof}
        We can write $E_0[e^{C\sum_{n=0}^\infty\mathbbm{1}\{\abs{Z_n}_1\leq r\} }]$ as
        \[
            \sum_{n=0}^\infty \frac{C^n}{n!} E_0\bigg[\bigg( \sum_{k=0}^\infty \mathbbm{1}_{\{ \abs{Z_k}_1\leq r \}} \bigg)^n\bigg]\leq \sum_{n=0}^\infty C^n \sum_{0\leq k_1\leq\cdots\leq k_n} P_0\bigg(\abs{Z_{k_1}}_1\leq r,\ldots,\,\abs{Z_{k_n}}_1\leq r\bigg).
        \]
        Invoking Lemma \ref{lem: bound for ball probability}, we can bound $E_0[e^{C\sum_{n=0}^\infty\mathbbm{1}\{\abs{Z_n}_1\leq r\} }]$ from above by
        \begin{equation*}\begin{aligned}
            &\sum_{n=0}^\infty C^n \sum_{0\leq k_1\leq\cdots\leq k_{n-1}} E_0\bigg[\abs{Z_{k_1}}_1\leq r,\ldots,\,\abs{Z_{k_{n-1}}}_1\leq r,\sum_{k_n=k_{n-1}}^\infty P_{Z_{k_{n-1}}} \big( \abs{Z_{k_n-k_{n-1}}}_1\leq r \big) \bigg]\\
            &\qquad\leq \sum_{n=0}^\infty C^n\eta(r) \sum_{0\leq k_1\leq\cdots\leq k_{n-1}} P_0\bigg(\abs{Z_{k_1}}_1\leq r,\ldots,\,\abs{Z_{k_{n-1}}}_1\leq r\bigg)\leq \sum_{n=0}^\infty C^n\eta(r)^n = \frac{1}{1-C\eta(r)},
        \end{aligned}\end{equation*}
        when $C\eta(r)<1$, which verifies the claim. 
    \end{proof}

    %lemma: function Lambda(beta)
    \begin{lemma}\label{lem: function Lambda(beta)}
        \normalfont
        Under the time-correlations, for any fixed time point $T>0$ there exists continuous function $\Lambda:[0,T]\to\mathbb{R}$ with $\exp(\Lambda(0))=1$ satisfying
        \[
            \frac{\mathbb{E}[e^{\beta\omega_x}e^{\beta\omega_y}\prod_{z\in I}e^{\beta\omega_z}]}{\mathbb{E}[e^{\beta\omega_x}\prod_{z\in I^\prime}e^{\beta\omega_z}]  \mathbb{E}[e^{\beta\omega_y}\prod_{z\in I^{\prime\prime}}e^{\beta\omega_z}] } \leq \exp(\Lambda(\beta)) \frac{\mathbb{E}[\prod_{z\in I}e^{\beta\omega_z}]}{\mathbb{E}[\prod_{z\in I^\prime}e^{\beta\omega_z}] \mathbb{E}[\prod_{z\in I^{\prime\prime}}e^{\beta\omega_z}] }
        \]
        for any $x,y\in\mathbb{N}\times\mathbb{Z}^d$ and $I,I^\prime,I^{\prime\prime}\subseteq(\mathbb{N}\times\mathbb{Z}^d)\backslash\{x,y\}$.
    \end{lemma}
    \begin{proof}
        We restrict to the case $\beta<1$ without loss of generality. For the $\beta\geq1$ case, we place some time point $T>\beta$ and $\beta/T$ when applying Jensen's inequality. Letting $\mathscr{F}_I\coloneqq\sigma(\omega_z:\,z\in I)$, we have    
        \[
            \mathbb{E}\big[e^{\beta\omega_x}e^{\beta\omega_y}\big|\mathscr{F}_I\big] \leq \mathbb{E}\big[e^{\omega_x+\omega_y}\big|\mathscr{F}_I\big]^\beta \leq \kappa^\beta \mathbb{E}\big[e^{\omega_x+\omega_y}\big]^\beta\leq\kappa_1(\beta) \mathbb{E}\big[e^{\beta\omega_x}e^{\beta\omega_y}\big],
        \]
        where
        \[
            \kappa_1(\beta)\coloneqq\sup\bigg\{\kappa^\beta \bigg(\frac{\mathbb{E}[e^{\omega_x+\omega_y}]^\beta}{\mathbb{E}[e^{\beta\omega_x}e^{\beta\omega_y}]}\bigg):\,\normx{x-y}_{L^1(\mathbb{N}\times\mathbb{Z}^d)}\leq r_0\bigg\},\qquad r_0\coloneqq\min\{\ell>d:\,\tfrac{\pi^2}{6}\ell^2 e^{-gt\ell}K_1\ell^2<1\}.
        \]
        Therefore, by Hölder's inequality we have
        \[
            \mathbb{E}\big[e^{\beta\omega_x}e^{\beta\omega_y}\prod_{z\in I}e^{\beta\omega_z}\big] \leq \mathbb{E}\big[e^{2\beta\omega_{1,0}}] \mathbb{E}\big[\prod_{z\in I}e^{\beta\omega_z}\big].
        \]
        On the other hand,
        \[
            \frac{1}{\mathbb{E}[e^{\beta\omega_x}|\mathscr{F}_{I^\prime}]} \leq \mathbb{E}\big[e^{-\beta\omega_x}\big|\mathscr{F}_{I^\prime}\big]\leq \kappa_2(\beta)^{1/2} \mathbb{E}\big[e^{-\beta\omega_{1,0}}],
        \]
        where
        \[
            \kappa_2(\beta)\coloneqq \kappa^{2\beta}\bigg(\frac{\mathbb{E}[e^{-\omega_{1,0}}]^\beta}{\mathbb{E}[e^{-\beta\omega_{1,0}}]}\bigg)^2.
        \]
        Collectively, with a similar estimate for $\mathbb{E}[e^{\beta\omega_y}|\mathscr{F}_{I^{\prime\prime}}]$, we get
        \begin{equation}\label{eqn: function Lambda}        
            \Lambda(\beta)\coloneqq \log \kappa_1(\beta)\kappa_2(\beta) + \log\mathbb{E}\big[e^{2\beta\omega_{1,0}}] + 2\log\mathbb{E}\big[e^{-\beta\omega_{1,0}}],
        \end{equation}
        which verifies the claim.
    \end{proof}   
    To ease notation, we will also write
    \[
        \Lambda_l(\beta)\coloneqq \log \kappa_1(\beta)\kappa_2(\beta) + \log\mathbb{E}\big[e^{2\beta\omega^l_{1,0}}] + 2\log\mathbb{E}\big[e^{-\beta\omega^l_{1,0}}]
    \]
    for each $l\geq1$. And thus $\Lambda_l(\beta)\xrightarrow[]{l\to\infty}\Lambda(\beta)$ for each $0<\beta\leq1$. Let $r_0\coloneqq\min\{\ell>d:\,\frac{\pi^2}{6}\ell^2 e^{-gt\ell}K_1\ell^2<1\}$ in $\mathbb{N}$. Via the extended Hölder's inequality for infinite products \cite{Karakostas}, we have
    \begin{equation}\label{eqn: martingale l2}
        \mathbb{E}E_Q\big[(\mathcal{H}_{n,\beta})^2\big] \leq E_QE^{S\otimes\widetilde{S}}_{0,0}\big[e^{\sum_{j=1}^n\frac{1}{2}V(\Vec{X}_j,\Vec{Y}_j)}\big] \leq E_Q E^{S\otimes\widetilde{S}}_{0,0} \big[ e^{\sum_{j=1}^\infty\frac{1}{2}V(\Vec{X}_j,\Vec{Y}_j) \mathbbm{1}\{d_1(\Vec{X}_j,\Vec{Y}_j)<r_0\} } \big]^{1/2}\prod_{k=r_0}^\infty A_k^{6/(\pi^2k^2)},
    \end{equation}
    where we write
    \[
        A_k\coloneqq E_QE^{S\otimes\widetilde{S}}_{0,0}\big[e^{\frac{\pi^2}{6}k^2e^{-gtk}\sum_{j=1}^\infty \mathbbm{1}\{d_1(\Vec{X}_j,\Vec{Y}_j=k)\} }\big],\qquad\forall~k\geq r_0.
    \]
    Using a version of Khas'minskii's lemma \cite[Proposition 4.1.1]{DaPrato/Zabczyk}, we get
    \begin{equation*}\begin{aligned}
        A_k &= \sum_{n=0}^\infty \frac{(\tfrac{\pi^2}{6}k^2 e^{-gtk})^n}{n!} E_QE^{S\otimes\widetilde{S}}_{0,0}\big[\big(\sum_{j=1}^\infty \mathbbm{1}\{d_1(\Vec{X}_j,\Vec{Y}_j)=k\} \big)^n\big]\\
        &\leq \sum_{n=0}^\infty (\tfrac{\pi^2}{6}k^2e^{-gtk})^n \sum_{1\leq j_1\leq\cdots\leq j_n} E_QE^{S\otimes\widetilde{S}}_{0,0}\big[\prod_{\ell=1}^n \mathbbm{1}\{d_1(\Vec{X}_{j_\ell},\Vec{Y}_{j_\ell}=k)\} \big].
    \end{aligned}\end{equation*}
    And then we have
    \begin{equation*}\begin{aligned}
        A_k &\leq \sum_{n=0}^\infty (\tfrac{\pi^2}{6}k^2e^{-gtk})^n \sum_{1\leq j_1\leq\cdots\leq j_{n-1}} E_QE^{S\otimes\widetilde{S}}_{0,0}\big[\prod_{\ell=1}^{n-1}\mathbbm{1}\{d_1(\Vec{X}_{j_\ell},\Vec{Y}_{j_\ell})=k\}\\
        &\quad\vdot \sum_{j_n\geq j_{n-1}} P^{S\otimes\widetilde{S}}_{S_{\tau_{j_{n-1}}},\widetilde{S}_{\tau_{j_{n-1}}}}\otimes Q\big(d_1(\Vec{X}_{j_n-j_{n-1}},\Vec{Y}_{j_n-j_{n-1}})\leq k\big) \big].
    \end{aligned}\end{equation*}
    Invoking Lemma \ref{lem: bound for ball probability}, we then have
    \begin{equation*}\begin{aligned}
        A_k &\leq \sum_{n=0}^\infty (\tfrac{\pi^2}{6}k^2 e^{-gtk})^nK_1k^2 \sum_{1\leq j_1\leq\cdots\leq j_{n-1}} E_QE^{S\otimes\widetilde{S}}_{0,0}\big[\prod_{\ell=1}^{n-1} \mathbbm{1}\{d_1(\Vec{X}_{j_\ell},\Vec{Y}_{j_\ell}=k)\} \big]\\
        &\leq \sum_{n=0}^\infty (\tfrac{\pi^2}{6}k^2e^{-gtk}K_1k^2)^n\leq\frac{1}{1-\frac{\pi^2}{6}k^2e^{-gtk}K_1k^2},\qquad\forall~k\geq r_0.
    \end{aligned}\end{equation*}
    Henceforth, we have
    \begin{equation}\label{eqn: ak}
        \prod_{k=r_0}^\infty A_k^{6/(\pi^2k^2)} \leq \prod_{k=r_0}^\infty \bigg(\frac{1}{1-\frac{\pi^2}{6}k^2e^{-gtk}K_1k^2}\bigg)^{6/(\pi^2k^2)}\leq K_2<\infty.
    \end{equation}
    Invoking Lemma \ref{lem: function Lambda(beta)}, now, we observe that
    \begin{equation}\begin{aligned}\label{eqn: next}
        &E_Q E^{S\otimes\widetilde{S}}_{0,0} \big[e^{\sum_{j=1}^\infty V(\Vec{X}_j,\Vec{Y}_j)\mathbbm{1}\{d_1(\Vec{X}_j,\Vec{Y}_j)<r_0\} }\big]\\ 
        &\qquad\leq E_Q E^{S\otimes\widetilde{S}}_{0,0} \big[e^{\sum_{j=1}^\infty 2\Lambda(\beta)N^{r_0}_j(\Vec{X},\Vec{Y})\mathbbm{1}\{d_1(\Vec{X}_j,\Vec{Y}_j)<r_0\} }\big]^{1/2}\prod_{k=r_0}^\infty B_k^{6/(\pi^2k^2)}
    \end{aligned}\end{equation}
    for some constant $K^\prime>0$, where the continuous function $\Lambda(\vdot)$ is defined by $\Lambda(\beta)\coloneqq\max\{2\Lambda_1(\beta),\Lambda_2(\beta)\}$, and we write
    \[
        B_k\coloneqq E_QE^{S\otimes\widetilde{S}}_{0,0}\big[ e^{\frac{\pi^2}{6}k^2 e^{-gtk} \sum_{j=1}^\infty \mathbbm{1}\{ d_1(\Vec{X}_j,\Vec{Y}_j)<r_0,\,d_1(\Vec{X}_j^\prime,\Vec{Y}_j^\prime)=k \} } \big],\qquad\forall~k\geq r_0.
    \]
    Here we use $\Vec{X}_j^\prime$ to denote the vertices of $\Vec{X}_j$ minus those at $\ell_1$-distance less than $r_0$ with $\Vec{Y}_j$. Following an almost identical argument to $A_k$, we derive
    \begin{equation}\label{eqn: dog a4}
        \prod_{k=r_0}^\infty B_k^{6/(\pi^2k^2)}\leq \prod_{k=r_0}^\infty \bigg(\frac{1}{1-\frac{\pi^2}{6}k^2e^{-gtk}K_1k^2}\bigg)^{6/(\pi^2k^2)}\leq K_2<\infty.
    \end{equation}
    Henceforth, by (\ref{eqn: next}) and (\ref{eqn: dog a4}),
    \begin{equation*}\begin{aligned}
        &E_Q E^{S\otimes\widetilde{S}}_{0,0} \big[ e^{\sum_{j=1}^\infty V(\Vec{X}_j,\Vec{Y}_j)\mathbbm{1}\{ d_1(\Vec{X}_j,\Vec{Y}_j)<r_0 \} } \big]^2\\
        &\qquad\leq K_2^2 E_Q E^{S\otimes\widetilde{S}}_{0,0}\big[ e^{\sum_{j=1}^\infty 2\Lambda_l(\beta)N^{r_0}_j(\Vec{X},\Vec{Y}) \mathbbm{1}\{ d_1(\Vec{X}_j,\Vec{Y}_j)<r_0 \} }\big]\\
        &\qquad\leq \sum_{n=0}^\infty K_2^2 \frac{2^n\Lambda_l(\beta)^n}{n!} E_Q E^{S\otimes\widetilde{S}}_{0,0} \bigg[ \bigg( \sum_{j=1}^\infty N^{r_0}_j(\Vec{X},\Vec{Y})\mathbbm{1}_{\{ d_1(\Vec{X}_j,\Vec{Y}_j)<r_0 \}} \bigg)^n \bigg].
    \end{aligned}\end{equation*}
    Therefore, by Lemma \ref{lem: probability estimate}, we have
    \begin{equation}\label{eqn: rest of ak}
        E_Q E^{S\otimes\widetilde{S}}_{0,0} \big[ e^{\sum_{j=1}^\infty V(\Vec{X}_j,\Vec{Y}_j)\mathbbm{1}\{ d_1(\Vec{X}_j,\Vec{Y}_j)<r_0 \} } \big]^2 \leq \frac{K_3}{1-2\Lambda_l(\beta)K_1r_0^2}<\infty
    \end{equation}
    whenever $\Lambda(\beta)K_1r_0^2<1/2$ and $l$ sufficiently large. Combining (\ref{eqn: ak}) and (\ref{eqn: rest of ak}), we resolve the estimate of (\ref{eqn: martingale l2}). And this verifies the assertion of Lemma \ref{lem: square-integrable martingale}.

%-----------------------------------------
\section{Technical lemmas}
    We extend \cite[Theorem 2.5, Lemma 3.2]{Comets/Shiga/Yoshida} which was established for the i.i.d. underlying field to the more general time-correlated field $\omega$. And we introduce the exponential factor $\Delta_{x,j}$ in the following refined proof to cancel the mixing nature of the correlated and non-i.i.d. environment.
    %lemma: how to estimate Z and hat Z ratio
    \begin{lemma}\label{lem: how to estimate Z and hat Z ratio}
        \normalfont
        With the same notational conventions as \textit{Step III.} of the proof of Theorem \ref{thm: quenched and annealed free energies exist}, we have
        \[
            \mathbb{E}\big[e^{2 \mathbb{E}[\text{log}(Z^{\beta,\omega}_N/\hat{Z}^{\beta,\omega}_{j,N})|\mathscr{K}_{j-1}] }\big]^{1/2} + \mathbb{E}\big[e^{-2 \mathbb{E}[\text{log}(Z^{\beta,\omega}_N/\hat{Z}^{\beta,\omega}_{j,N})|\mathscr{K}_{j-1}] }\big]^{1/2} \leq C(\beta),\qquad\forall~j\leq N,
        \]
        for some positive constant $C(\beta)$.
    \end{lemma}
    \begin{proof}
        We first notice that 
        \begin{equation*}
            \mathbb{E}\big[e^{2 \mathbb{E}[\text{log}(Z^{\beta,\omega}_N/\hat{Z}^{\beta,\omega}_{j,N})|\mathscr{K}_{j-1}] }\big] \leq \mathbb{E}\big[e^{2 \text{log}\mathbb{E}[(Z^{\beta,\omega}_N/\hat{Z}^{\beta,\omega}_{j,N})|\mathscr{K}_{j-1}] }\big],
        \end{equation*}
        where 
        \[
            \mathbb{E}\big[(Z^{\beta,\omega}_N/\hat{Z}^{\beta,\omega}_{j,N})\big|\mathscr{K}_{j-1}\big] \leq \mathbb{E}\big[\sum_{x\in\mathbb{Z}^d} \alpha_{x,j} e^{\beta\omega_{j,x}-\Delta_{j,x}} \big|\mathscr{K}_{j-1}\big] \underset{\forall~(n,z)\in\mathbb{N}\times\mathbb{Z}^d}{\leq} \mathbb{E}\big[e^{2\beta\omega_{n,z}}\big]^{1/2} \mathbb{E}\big[\sum_{x\in\mathbb{Z}^d} \alpha_{x,j} \big|\mathscr{K}_{j-1}\big].
        \]
        Therefore,
        \begin{equation}\label{eqn: b1}
            \mathbb{E}\big[e^{2 \mathbb{E}[\text{log}(Z^{\beta,\omega}_N/\hat{Z}^{\beta,\omega}_{j,N})|\mathscr{K}_{j-1}] }\big]^{1/2} \leq \mathbb{E}\big[e^{2\beta\omega_{n,z}}\big]^{1/2} \leq C^\prime(\beta)<\infty.
        \end{equation}
        On the other hand, for each $u>-1$, we let $\psi(u)\coloneqq u- \log(1+u)$. In particular,
        \[
            -\mathbb{E}\big[(Z^{\beta,\omega}_N/\hat{Z}^{\beta,\omega}_{j,N})\big|\mathscr{K}_{j-1}\big] \leq 1+ \mathbb{E}\big[ \psi(U) \big|\mathscr{K}_{j-1}\big],\qquad\text{where}\quad U\coloneqq\sum_{x\in\mathbb{Z}^d}\alpha_{x,j} e^{\beta\omega_{j,x}-\Delta_{j,x}} -1.
        \]
        Now we fix some sufficiently small $0<\epsilon<1$ such that $\log\epsilon\leq-1$. It is then obvious that
        \[
            \mathbb{E}\big[\psi(U)\big|\mathscr{K}_{j-1}\big] \leq \mathbb{E}\big[\psi(U),\,1+U\geq\epsilon\big|\mathscr{K}_{j-1}\big] -  \mathbb{E}\big[\log(1+U),\,1+U\leq\epsilon\big|\mathscr{K}_{j-1}\big]
        \]
        since $U\leq\epsilon-1$ conditioned on $\{1+U\leq\epsilon\}$. Note that $\psi(U)\leq\frac{1}{2}(u/\epsilon)^2$ as long as $1+u\geq\epsilon$. So,
        \[
            2\epsilon^2\mathbb{E}\big[\psi(U),\,1+U\geq\epsilon\big|\mathscr{K}_{j-1}\big] \leq \mathbb{E}\big[U^2\big|\mathscr{K}_{j-1}\big] \leq 1+\sum_{x\in\mathbb{Z}^d} \alpha_{x,j}^2 \mathbb{E}\big[ e^{2\beta\omega_{j,x}-2\Delta_{j,x}} \big|\mathscr{K}_{j-1}\big].
        \]
        Then, for any $(n,z)\in\mathbb{N}\times\mathbb{Z}^d$,
        \begin{equation}\label{eqn: b2}
            2\epsilon^2\mathbb{E}\big[\psi(U),\,1+U\geq\epsilon\big|\mathscr{K}_{j-1}\big] \leq 1 + \mathbb{E}\big[ e^{2\beta\omega_{j,x}} \big] \sum_{x\in\mathbb{Z}^d}\alpha_{x,j}^2 \leq \Bar{C}^\prime(\beta)<\infty,\qquad\mathbb{P}\text{-a.s.,}
        \end{equation}
        by the fact that $\sum_{x\in\mathbb{Z}^d}\alpha_{x,j}^2\leq1$. Moreover, we have the following relations,
        \[
            \big\{1+U\leq\epsilon\big\}\subseteq \big\{-V\leq(1+U)\leq\log\epsilon\big\}\subseteq\big\{\log(1+U)\leq V\big\}\cap\big\{1\leq V\big\}
        \]
        with $V\coloneqq\sum_{x\in\mathbb{Z}^d}\alpha_{x,j}(\Delta_{j,x}-\beta\omega_{j,x})$. Henceforth,
        \[
            -\mathbb{E} \big[\log(1+U),\,1+U\leq\epsilon\big|\mathscr{K}_{j-1}\big] \leq \mathbb{E}\big[V,\,1\leq V\big|\mathscr{K}_{j-1}\big]\leq \mathbb{E}\big[e^V\big|\mathscr{K}_{j-1}\big] \leq \sum_{x\in\mathbb{Z}^d} \alpha_{x,j} \mathbb{E}\big[e^{\beta\omega_{j,x}-\Delta_{j,x}}\big|\mathscr{K}_{j-1}\big],
        \]
        where the last inequality is due to Jensen's inequality. Hence for any $(n,z)\in\mathbb{N}\times\mathbb{Z}^d$,
        \begin{equation}\label{eqn: b3}                  
            -\mathbb{E} \big[\log(1+U),\,1+U\leq\epsilon\big|\mathscr{K}_{j-1}\big] \leq\sum_{x\in\mathbb{Z}^d} \alpha_{x,j} \mathbb{E}\big[ e^{2\beta\omega_{j,x}-2\Delta_{j,x}}  \big|\mathscr{K}_{j-1} \big]^{1/2}\leq \mathbb{E}\big[e^{2\beta\omega_{n,z}}\big]^{1/2} \sum_{x\in\mathbb{Z}^d} \alpha_{x,j} \leq \Bar{C}^{\prime\prime}(\beta), 
        \end{equation}
        $\mathbb{P}$-a.s. Combining (\ref{eqn: b2}) and (\ref{eqn: b3}), we thus get
        \begin{equation}\label{eqn: b4}
            \mathbb{E}\big[e^{-2 \mathbb{E}[\text{log}(Z^{\beta,\omega}_N/\hat{Z}^{\beta,\omega}_{j,N})|\mathscr{K}_{j-1}] }\big]^{1/2} \leq \mathbb{E}\big[e^{2+\epsilon^{-2}\Bar{C}^\prime(\beta)+\Bar{C}^{\prime\prime}(\beta)}\big]^{1/2}\leq C^{\prime\prime}(\beta)<\infty.
        \end{equation}
        Viewing (\ref{eqn: b1}) and (\ref{eqn: b4}) and letting $C(\beta)\coloneqq C^\prime(\beta)+C^{\prime\prime}(\beta)$, we have proved the assertion.
    \end{proof}

    For any Borel measurable function $f:\mathbb{R}\to\mathbb{R}$ and finite time $T>0$, we use $\normx{f}_T\coloneqq\esssup\{\abs{f(t)}:\,0\leq t\leq T\}$ to denote its $L^\infty$-norm up to time $T$.

    %lemma: L2 space for processes L and W
    \begin{lemma}\label{lem: L2 space for processes L and W}
        \normalfont
        For each $n\in\mathbb{N}$ and $T>0$, we have
        \[
            \mathbb{E}\big[\normy{(W^{\beta,\omega}_N)^{-1}}_T^2\big]<\infty\qquad\text{and}\qquad \mathbb{E}\big[\normy{\frac{\partial}{\partial\beta}\log W^{\beta,\omega}_N}_T^2\big]<\infty,
        \]
        where the $L^\infty$-norm is with respect to $\beta\geq0$.
    \end{lemma}
    \begin{proof}
        We show the $L^2$-boundedness. By Jensen's inequality, it is clear that
        \[
            (W^{\beta,\omega}_N)^{-1} \leq \mathbb{E} \big[Z^{\beta,\omega}_N\big] \vdot E^S_0\big[e^{\beta\sum_{k=1}^N \omega_{k,S_k} }\big],\qquad\forall~n\in\mathbb{N}\quad\text{and}\quad\beta\geq0.
        \]
        Invoking Lemma \ref{lem: kappa-bound for a cone}, we get
        \begin{equation}\begin{aligned}\label{eqn: L2 one}                    
            \mathbb{E}\big[\normy{(W^{\beta,\omega}_N)^{-1}}_T^2\big]^{1/2} &\leq \kappa^N \mathbb{E}\big[e^{T\omega_{n,z}}\big]\vdot \mathbb{E}\big[E^S_0\big(e^{T\sum_{k=1}^N \omega_{k,S_k} }\big)^2\big]^{1/2}\\
            &\leq \kappa^{2N} \mathbb{E}\big[e^{2T\omega_{n,z}}\big] <\infty,\qquad\forall~(n,z)\in\mathbb{N}\times\mathbb{Z}^d.
        \end{aligned}\end{equation}
        On the other hand, for each $\beta\geq0$, 
        \[
            \mathbb{E} [Z^{\beta,\omega}_N] \frac{\partial}{\partial\beta}W^{\beta,\omega}_N = E^S_0\big[\big(\sum_{k=1}^N \omega_{k,S_k} - \frac{1}{\mathbb{E}[Z^{\beta,\omega}_N]} \frac{\partial}{\partial\beta} \mathbb{E}[Z^{\beta,\omega}_N]  \big) e^{\beta \sum_{k=1}^N \omega_{k,S_k} } \big].
        \]
        In particular, we have
        \[
            \frac{\partial}{\partial\beta} \mathbb{E}\big[Z^{\beta,\omega}_N\big]  = \mathbb{E} E^S_0\big[ \big(\sum_{k=1}^N \omega_{k,S_k}\big) e^{\beta\sum_{k=1}^N \omega_{k,S_k} } \big],\qquad\forall~\beta\geq0.
        \]
        Henceforth, $N^{-1} \mathbb{E}[\normx{\frac{\partial}{\partial\beta} \log W^{\beta,\omega}_N }_T^2]$ is less than or equal to
        \[
             E_{\mathbb{P}\otimes Q} E^S_0\bigg[\sum_{\ell=1}^N \bigg(  \omega_{\ell,S_\ell} - \frac{1}{\mathbb{E}[Z^{T,\omega}_N]} \mathbb{E} E^S_0[ \omega_{\ell,S_\ell} e^{T\sum_{k=1}^N \omega_{k,S_k} } ] \bigg)^2 e^{2T\sum_{k=1}^N \omega_{k,S_k} } \bigg].
        \]
        Therefore, invoking Lemma \ref{lem: tau mixing inequality} again, we observe that $ \mathbb{E}[\normx{\frac{\partial}{\partial\beta} \log W^{\beta,\omega}_N }_T^2]^{1/2}$ is less than or equal to
        \begin{equation*}
            \quad N^{1/2}\kappa^{2N} C(\beta)\vdot \mathbb{E} E^S_0\bigg[\sum_{\ell=1}^N \bigg(  \omega_{\ell,S_\ell} - \frac{1}{\mathbb{E}[e^{T\omega_{\ell,S_\ell}}]} \mathbb{E} E^S_0[ \omega_{\ell,S_\ell} e^{T \omega_{\ell,S_\ell} } ] \bigg)^2 e^{2T\omega_{\ell,S_\ell} } \bigg]^{1/2},
        \end{equation*}
        Henceforth,
        \begin{equation}\label{eqn: L2 two}
            \mathbb{E}\big[\normx{\frac{\partial}{\partial\beta} \log W^{\beta,\omega}_N }_T^2\big]^{1/2} \leq  N^{1/2}\kappa^{2N} C(\beta)\vdot \mathbb{E}\big[Z^{3T,\omega}_N\big]^{1/2}<\infty.
        \end{equation}
        And the assertion is verified.
    \end{proof}

    One should observe that, as a continuous convex function, the annealed free energy $\beta\mapsto\lambda(\beta)$ admits left-continuous left-derivative $\lambda^\prime_-(\beta)$ and right-continuous right-derivative $\lambda^\prime_+(\beta)$ at each $\beta\geq0$. Of interest in its own right, we establish the (Gâteaux) differentiability of $\lambda(\vdot)$, showing that $\lambda^\prime_-(\beta)=\lambda^\prime_+(\beta)$ for each $\beta$. Notice that the notion of Fréchet and Gâteaux differentials \cite{Zorn} are identical on the real line.
    
    %lemma: annealed free energy is differentiable.
    \begin{lemma}\label{lem: annealed free energy is differentiable}
        \normalfont
        The limiting annealed free energy $\beta\mapsto\lambda(\beta)$ is Gâteaux differentiable on $[0,\infty)$.
    \end{lemma}
    \begin{proof}
        We divide the proof into several steps.\par\noindent
        \textbf{Step I.} Equi-continuity of first derivatives.\par\noindent
        For each $N\in\mathbb{N}$ and $\beta\geq0$, we denote $\lambda_N(\beta)=\frac{1}{N}\log\mathbb{E}[Z^{\beta,\omega}_N]$. And in \textit{Step I.} of the proof of Theorem \ref{thm: quenched and annealed free energies exist} we have shown that $\lambda_N(\beta)\to\lambda(\beta)$ as $N\to\infty$ for each $\beta\geq0$, via an large deviation argument. It is easily observed that each $\lambda_N(\vdot)\in C^\infty(\mathbb{R})$ because of (\ref{eqn: finit exponential moment of omega}). Hence, we can write
        \[
            \lambda^\prime_N(\beta)=\frac{1}{N}\frac{\mathbb{E}E^S_0[(\sum_{k=1}^N\omega_{k,S_k})I_N]}{\mathbb{E}[Z^{\beta,\omega}_N]},\quad\lambda^{\prime\prime}_N(\beta)=\frac{1}{N}\bigg(\frac{\mathbb{E}E^S_0[(\sum_{k=1}^N\omega_{k,S_k})^2I_N] }{\mathbb{E}[Z^{\beta,\omega}_N]} - \frac{\mathbb{E}E^S_0[(\sum_{k=1}^N\omega_{k,S_k})I_N]^2}{\mathbb{E}[Z^{\beta,\omega}_N]^2} \bigg)
        \]
        for all $\beta\geq0$, where we adopt the following convention to ease notation,
        \[
            I_N\coloneqq e^{\beta\sum_{k=1}^N\omega_{k,S_k}},\quad I_{i,j}\coloneqq e^{\beta(\omega_{i,S_i}+\omega_{j,S_j})},\quad\text{and}\quad I_N^{i,j}\coloneqq e^{\beta\sum_{k\neq i,j}\omega_{k,S_k}},\qquad\forall~1\leq i,j\leq N.
        \]
        Therefore, we can express $\lambda_N^{\prime\prime}$ as 
        \begin{equation*}\begin{aligned}
            \lambda^{\prime\prime}_N(\beta) &= \sum_{i,j=1}^N \frac{1}{N}\bigg(\frac{\mathbb{E}E^S_0[\omega_{i,S_i}\omega_{j,S_j}I_N]\mathbb{E}E^S_0[I_N]}{\mathbb{E}[Z^{\beta,\omega}_N]^2} - \frac{\mathbb{E}E^S_0[\omega_{i,S_i}I_N]\mathbb{E}E^S_0[\omega_{j,S_j} I_N]}{\mathbb{E}[Z^{\beta,\omega}_N]^2} \bigg)\\
            &= \sum_{i,j=1}^N \mu_{i,j}^2 \frac{1}{N} \frac{ \mathbb{E}E^S_0\otimes \Tilde{\mathbb{E}}E^{\Tilde{S}}_0 [(\omega_{i,S_i}\omega_{j,S_j}I_N)\Tilde{I}_N -(\omega_{i,S_i}I_N)\Tilde{\omega}_{j,S_j}\Tilde{I}_N ] }{ \mathbb{E}E^S_0[I_{i,j}]^2 \mathbb{E}E^S_0[I_N^{i,j}]^2 }.
        \end{aligned}\end{equation*}
        Here we use $\Tilde{\omega}$ as an independent copy of $\omega$ under law $\Tilde{\mathbb{E}}$, and $(\Tilde{S}_n)_{n\geq1}$ under law $P^{\Tilde{S}}_0$ of $(S_n)_{n\geq1}$. And $\Tilde{I}_N$ is the respective counterpart of $I_N$ using $\Tilde{\omega}$ and $(\Tilde{S}_n)_{n\geq1}$. The factor $\mu_{i,j}$ comes from a correlation in the denominator, and has the estimate $1/\kappa\leq\mu_{i,j}\leq\kappa$ by Lemma \ref{lem: kappa-bound for a cone}. Hence,
        \[
            \lambda^{\prime\prime}_N(\beta) = \sum_{i,j=1}^N \mu_{i,j}^2\nu_{i,j}^2\frac{1}{N} \frac{ \mathbb{E}E^S_0\otimes \Tilde{\mathbb{E}}E^{\Tilde{S}}_0 [(\omega_{i,S_i}\omega_{j,S_j}I_{i,j})\Tilde{I}_{i,j} -(\omega_{i,S_i}I_{i,j})\Tilde{\omega}_{j,S_j}\Tilde{I}_{i,j}]\vdot \mathbb{E}E^S_0[I_N^{i,j}]^2  }{\mathbb{E}E^S_0[I_{i,j}]^2\vdot \mathbb{E}E^S_0[I_N^{i,j}]^2},
        \]
        where $\nu_{i,j}$ comes from a correlation in the nominator and has the estimate $1/\kappa\leq\nu_{i,j}\leq\kappa$ by Lemma \ref{lem: kappa-bound for a cone}. Let $P^{i,j}_N$ be the law absolutely continuous with respect to $\mathbb{E}P^S_0$ with Radon--Nikodým derivative
        \[
            \frac{dP^{i,j}_N}{d\mathbb{E}P^S_0} = \frac{I_{i,j}}{\mathbb{E}E^S_0[I_{i,j}]},\qquad\forall~1\leq i,j\leq N.
        \]
        We also denote $E^{i,j}_N\coloneqq E_{P^{i,j}_N}$. Then,
        \begin{equation}\label{eqn: b7}
            \lambda_N^{\prime\prime}(\beta) = \sum_{i,j=1}^N \mu_{i,j}^2\nu_{i,j}^2\frac{1}{N} \big( E^{i,j}_N[\omega_{i,S_i}\omega_{j,S_j}] -E^{i,j}_N[\omega_{i,S_i}] E^{i,j}_N[\omega_{j,S_j}]  \big).
        \end{equation}
        Hence, by the time-correlation $(\textbf{TC})_{C,g}$ or $(\textbf{TCG})_{C,g}$, and following similar steps as Lemma \ref{lem: kappa-bound for a cone}, we get
        \[
            \big| E^{i,j}_N[\omega_{i,S_i}\omega_{j,S_j}] - \hat{\mathbb{E}}[\omega_{1,0}]^2 \big| \leq \hat{\mathbb{E}}[\omega_{1,0}]^2\big(\text{exp}(e^{-gt\abs{i-j}})-1\big)
        \]
        as well as
        \[
            \big| E^{i,j}_N[\omega_{i,S_i}] E^{i,j}_N[\omega_{j,S_j}] - \hat{\mathbb{E}}[\omega_{1,0}]^2 \big| \leq \hat{\mathbb{E}}[\omega_{1,0}]^2\big(\text{exp}(e^{-gt\abs{i-j}})-1\big),
        \]
        for some constant $0<t<1$ and for the law $\hat{\mathbb{P}}$ defined by $d\hat{\mathbb{P}}/d\mathbb{P} = \frac{1}{\mathbb{E}[e^{\beta\omega_{1,0}}]}e^{\beta\omega_{1,0}}$. Therefore, by (\ref{eqn: b7}),
        \begin{equation*}\begin{aligned}
            \abs{\lambda^{\prime\prime}_N(\beta)} &\leq C \hat{\mathbb{E}}[\omega_{1,0}]^2 \frac{1}{N} \sum_{i,j=1}^N \big(\text{exp}(e^{-gt\abs{i-j}})-1\big) \leq C^\prime \hat{\mathbb{E}}[\omega_{1,0}]^2 \frac{1}{N} \sum_{l=1}^N \sum_{\abs{i-j}=l} e^{-gtl}\\
            &\leq C^{\prime\prime} \hat{\mathbb{E}}[\omega_{1,0}]^2 \frac{1}{N} \sum_{l=1}^N 2Ne^{-gtl} \leq K(\beta),\qquad\forall~\beta\geq0\quad\text{and}\quad N\geq1.
        \end{aligned}\end{equation*}
        Henceforth, we know the first derivatives $\lambda^\prime_N(\beta)$ are equi-continuous at each $\beta\geq0$.\par\noindent
        \textbf{Step II.} Proof of differentiability.\par\noindent
        In light of \cite[Theorem 2.1]{Zagrodny}, for any fixed $\beta\geq0$ we can find two sequence $\{\beta_N\}_{N\geq1}$, $\{\beta^\prime_N\}_{N\geq1}$ such that $\beta_N\to\beta$ and $\beta^\prime_N\to\beta$ as $N\to\infty$. Moreover, $\lambda^\prime_N(\beta_N)\to\lambda^\prime_-(\beta)$ and $\lambda^\prime_N(\beta_N)\to\lambda^\prime_+(\beta)$ as $N\to\infty$. Hence,
        \[
            \lambda^\prime_-(\beta)-\lambda^\prime_+(\beta) = \lim_{N\to\infty} \lambda^\prime_N(\beta_N) -  \lambda^\prime_N(\beta^\prime_N).
        \]
        Because of the equi-continuity of $\{\lambda^\prime_N\}_{N\geq1}$, for any $\epsilon>0$ there exists $\delta>0$ such that whenever $\abs{\beta^\prime-\beta}<\delta$ for the fixed $\beta$, we have $\abs{\lambda^\prime_N(\beta^\prime) -  \lambda^\prime_N(\beta)}<\epsilon/2$ for all $N\geq1$. Therefore, we choose $N_0\in\mathbb{N}$ such that $\abs{\beta_N-\beta}<\delta$ and $\abs{\beta^\prime_N-\beta}<\delta$ for all $N\geq N_0$. Thus, 
        \[
             \abs{\lambda^\prime_-(\beta)-\lambda^\prime_+(\beta)} \leq \varlimsup_{N\to\infty} \abs{\lambda^\prime_N(\beta_N) -  \lambda^\prime_N(\beta^\prime_N)}\leq\epsilon.
        \]
        Letting $\epsilon\to0$, we have $\lambda^\prime_-(\beta)=\lambda^\prime_+(\beta)$ for all $\beta\geq0$, yielding the Gâteaux differentiability for $\lambda(\vdot)$.\par\noindent
        \textbf{Step III.} Pointwise convergence.\par\noindent
        Since $\lambda(\vdot)$ is convex differentiable on $[0,\infty)$, at each $\beta\geq0$, for any $\epsilon>0$ there exists $h>0$ such that
        \[
            \lambda^\prime(\beta-\epsilon)<\frac{\lambda(\beta)-\lambda(\beta-h)}{h}\leq\frac{\lambda(\beta+h)-\lambda(\beta)}{h}<\lambda^\prime(\beta)+\epsilon.
        \]
        Then, there exists $N_0\in\mathbb{N}$ such that for any $N\geq N_0$,
        \[
            \lambda^\prime(\beta-\epsilon)<\frac{\lambda_N(\beta)-\lambda_N(\beta-h)}{h}\leq \lambda^\prime_N(\beta) \leq\frac{\lambda_N(\beta+h)-\lambda_N(\beta)}{h}<\lambda^\prime(\beta)+\epsilon.
        \]
        And this yields
        \begin{equation}\label{eqn: pointwise convergence of derivative of lambda}
            \lambda^\prime_N(\beta)\xrightarrow[]{N\to\infty}\lambda^\prime(\beta),\qquad\forall~\beta\geq0,
        \end{equation}
        verifying the claim.
    \end{proof}

    Knowing that the limiting annealed free energy is differentiable also allows us to conclude that the large deviation rate function $I(\vdot)$ mentioned at the end of Section \ref{sec: regularity} is strictly convex, see \cite[Theorem 26.3]{Rockafellar}. When the underlying field $\omega$ is i.i.d. in both space and time, this fact follows immediately from \cite[Theorem 2]{Vysotsky} and the simple expression of $\lambda(\vdot)$. But when there is time-correlation, the subtleties of correlated structure have to be taken carefully. We also encourage readers to \cite{Stromberg} for more internal convexity properties of the free energy $\lambda$.\par
    We also need to specify how exactly the process $(\mathcal{H}_{n,\beta})_{n\geq1}$ is constructed at the beginning of Section \ref{sec: localization}. Indeed, writing each $\tau_n=\tau^{(L)}_n$, we recall that it has been defined $\mathcal{L}_{n,\beta}(\omega,\epsilon)=\frac{Z^{1,\xi(l)}_{\tau_n}}{\mathbb{E}E_Q[Z^{1,\xi(l)}_{\tau_n}]}$ for each $n\geq1$. And to cancel the time-correlation effect from distant $\omega$-coordinates, we whence define
    \begin{equation}\label{eqn: definition of H}    
        \mathcal{H}_{n,\beta}(\omega,\epsilon)\coloneqq\frac{E^S_0[\prod_{j=0}^{n-1} e^{\Delta_{j,n}}e^{\sum_{k=\tau_{j}+1}^{\tau_{j+1}} \xi^l_{k,S_k} } ]}{\mathbb{E}\overline{E}^S_0[\prod_{j=0}^{n-1} e^{\Delta_{j,n}}e^{\sum_{k=\tau_{j}+1}^{\tau_{j+1}} \xi^l_{k,S_k} }]},\qquad\forall~n\geq1,
    \end{equation}
    with each $\Delta_{j,n}\coloneqq\log \mathbb{E}\overline{E}^S_0[e^{\sum_{k=\tau_j+1}^{\tau_{j+1}} \xi^l_{k,S_k} } |\mathscr{G}_{j} ] - \log \mathbb{E}\overline{E}^S_0[e^{\sum_{k=\tau_j+1}^{\tau_{j+1}} \xi^l_{k,S_k} }  ] $.
    %lemma: mixing inequality for martingale H
    \begin{lemma}\label{lem: mixing inequality for martingale H}
        \normalfont
        For each $\beta>0$ and $n\geq1$, we have
        \[
            C\exp(-n e^{-gtL}) \mathcal{H}_{n,\beta}(\omega,\epsilon)\leq\mathcal{L}_{n,\beta}(\omega,\epsilon)\leq C^\prime\exp(n e^{-gtL}) \mathcal{H}_{n,\beta}(\omega,\epsilon),\qquad\mathbb{P}\otimes Q\text{-a.s.,}
        \]
        with constants $C,C^\prime>0$.
    \end{lemma}
    \begin{proof}
        It is obvious that $\Delta_{1,n}=0$, implying $\mathcal{H}_{1,\beta}=\mathcal{L}_{1,\beta}$. Moreover, for each $1<j\leq n$,
        \[
            \exp(-e^{-gtL})\leq\exp(\Delta_{j,n})=\frac{\mathbb{E}\overline{E}^S_0[e^{\sum_{k=\tau_j+1}^{\tau_{j+1}} \xi^l_{k,S_k} } |\mathscr{G}_{j} ]}{\mathbb{E}\overline{E}^S_0[e^{\sum_{k=\tau_j+1}^{\tau_{j+1}} \xi^l_{k,S_k} } ]}\leq \exp(e^{-gtL}),
        \]
        by Lemma \ref{lem: tau mixing inequality}, which iteratively verifies the assertion.
    \end{proof}
    Using the (relatively) informal language the the beginning of Section \ref{sec: localization}, it is equivalent to write $\mathcal{H}_{n,\beta}=\mathcal{L}_{n,\beta}(\omega^*,\epsilon)$, where $\xi^{l,*}_{k,S_k}=(\tau_{j+1}-\tau_j)^{-1}\Delta_{j,n}+\xi^l_{k,S_k}$ for all $\tau^{(L)}_{j}+1\leq k\leq\tau^{(L)}_{j+1}$.\par
    Remember that in Section \ref{sec: localization} we have defined the truncated free energies $\rho_l(\beta)=\lim_{N\to\infty} N^{-1}\log Z^{\beta,\omega(l)}_N$ and $\lambda_l(\beta)=\lim_{N\to\infty} N^{-1}\log \mathbb{E}[Z^{\beta,\omega(l)}_N]$ for each index $l\geq1$. The following lemma gives their respective pointwise limits.

    %lemma: rho_l to rho, lambda_l to lambda
    \begin{lemma}\label{lem: rho_l to rho, lambda_l to lambda}
        \normalfont
        Under the time-correlations $(\textbf{TC})_{C,g}$ or $(\textbf{TCG})_{C,g}$, for any $\beta\geq0$ we have
        \[
            \lim_{l\to\infty}\rho_l(\beta)=\rho(\beta) \qquad \text{and} \qquad \lim_{l\to\infty}\lambda_l(\beta)=\lambda(\beta).
        \]
    \end{lemma}
    \begin{proof}
        For any $l\geq1$ and $N\geq1$, we have
        \[
            \sum_{k=1}^N \omega^l_{k,S_k} - \sum_{k=1}^N \omega_{k,S_k} = \sum_{k=1}^N \abs{\omega_{k,S_k}+l}\mathbbm{1}_{\{\omega_{k,S_k}+l\leq0\}}.
        \]
        And whence, for each $\delta>0$ and prime $p>1$,
        \begin{equation*}\begin{aligned}
            &\mathbb{P}\bigg( \sup_{S\in\mathcal{S}_N} \sum_{k=1}^N \abs{\omega_{k,S_k}+l}\mathbbm{1}_{\{\omega_{k,S_k}+l\leq0\}} \geq \delta N \bigg) \leq
            e^{-\delta N} \mathbb{E} \bigg[\sup_{S\in\mathcal{S}_N} e^{\sum_{k=1}^N \abs{\omega_{k,S_k}+l}\mathbbm{1}_{\{\omega_{k,S_k}+l\leq0\}}} \bigg]\\
            &\qquad\leq e^{-\delta N} \mathbb{E} \bigg[\sup_{S\in\mathcal{S}_N} e^{\sum_{k=1}^{N-p} \abs{\omega_{k,S_k}+l}\mathbbm{1}_{\{\omega_{k,S_k}+l\leq0\}}} \mathbb{E}\big[ e^{\sum_{k=N-p+1}^N \abs{\omega_{k,S_k}+l}\mathbbm{1}_{\{\omega_{k,S_k}+l\leq0\}}} \big|\mathscr{F}_{I^S_{N-p}}\big] \bigg]\\
            &\qquad\leq e^{-\delta N} \kappa^{N/p} \bigg(1+\sum_{k=1}^p\frac{p!}{(p-k)!k!}\sup_{S\in\mathcal{S}_p}\mathbb{E}\big[e^{\sum_{j=1}^k \abs{\omega_{j,S_j}+l} },\,\omega_{j,S_j}+l\leq0:\,j\leq k\big]\bigg)^{N/p},
        \end{aligned}\end{equation*}
        where $\mathscr{F}_{I^S_{N-p}}\coloneqq\sigma(\omega_{k,S_k}:\,k=1,\ldots,N-p)$ and the last inequality is due to Lemma \ref{lem: kappa-bound for a cone}. Recursively, we know the above express is no larger than
        \begin{equation*}\begin{aligned}
             &e^{-\delta N} \kappa^{N/p} \bigg(1+\sum_{k=1}^p\frac{p!}{(p-k)!k!} \sup_{S\in\mathcal{S}_{p-1}} \mathbb{E}\big[e^{4\sum_{j=1}^{k-1} |\omega_{j,S_j}| }\big]^{1/4}\mathbb{E}\big[e^{4|\omega_{1,0}|}\big]^{1/4}\mathbb{P}(\omega_{1,0}\leq-l)^{1/2}\bigg)^{N/p}\\
             &\qquad\leq e^{-\delta N}\kappa^{N/p}(1+C(p)e^{-l/2})^{N/p}
        \end{aligned}\end{equation*}
        for some constant $C(p)>0$ depending only on the prime $p$. Choose $p$ large enough with $p^{-1}\log \kappa<\delta$ and then choose $l>l_0$ for some $l_0=l_0(\delta)$. Then, via Borel--Cantelli lemma,
        there exists some $N_0(\omega)\in\mathbb{N}$, independent of $dP^S_0$, such that 
        \begin{equation}\label{eqn: dog b9}        
            \sup_{S\in\mathcal{S}_N}\sum_{k=1}^N \abs{\omega_{k,S_k}+l}\mathbbm{1}_{\{\omega_{k,S_k}+l\leq0\}} \leq \delta N,\qquad\forall~N\geq N_0,\qquad\mathbb{P}\text{-a.s.}
        \end{equation}
        Therefore, whenever $N\geq N_0(\omega)$ and $l>l_0(\delta)$,
        \[
            Z^{\beta,\omega(l)}_N = \sum_{S^\prime\in \mathcal{S}_N} P^S_0(S_j=S^\prime_j,\,j\leq N)e^{\beta\sum_{k=1}^N \omega^l_{k,S_k} } \leq e^{\delta N} Z^{\beta,\omega}_N = \sum_{S^\prime\in \mathcal{S}_N} P^S_0(S_j=S^\prime_j,\,j\leq N)e^{\beta\sum_{k=1}^N \delta + \omega_{k,S_k} }.
        \]
        Taking logarithm and then letting $N\to\infty$, $l\to\infty$, and $\delta\to0$, we get
        \[
            \lim_{N\to\infty}\frac{1}{N}\log  Z^{\beta,\omega}_N \leq \lim_{l\to\infty}\lim_{N\to\infty}\frac{1}{N}\log  Z^{\beta,\omega(l)}_N \leq \lim_{\delta\to0} \lim_{N\to\infty} \delta + \frac{1}{N}\log  Z^{\beta,\omega}_N,
        \]
        which obviously implies $\rho_l(\beta)\xrightarrow[]{l}\rho(\beta)$ for all $\beta\geq0$. On the other hand, when $\lambda(2\beta)-2\lambda(\beta)\leq...$, we know 
        \[
            \rho(\beta) \leq \lambda(\beta)\leq \lim_{l\to\infty}\lambda_l(\beta) = \rho(\beta).
        \]
        Hence, it is then obvious that $\lambda_l(\beta)\xrightarrow[]{l}\lambda(\beta)$ as well, verifying the assertion.
    \end{proof}

\bibliographystyle{plain}
%\bibliography{literature}
\begin{spacing}{1}

\end{spacing}

\end{document}